\documentclass[12pt]{amsart}
\usepackage{amsmath,amssymb,amsthm,amsfonts,enumitem,graphicx,xcolor,subcaption,tikz,bbm}
\usepackage{tikz-3dplot}
\newtheorem{theorem}{Theorem}[section]
\newtheorem{definition}[theorem]{Definition}

\newtheorem{proposition}[theorem]{Proposition}
\newtheorem{lemma}[theorem]{Lemma}
\newtheorem{corollary}[theorem]{Corollary}
\newtheorem{remark}[theorem]{Remark}
\newtheorem{example}[theorem]{Example}
\newtheorem{algorithm}[theorem]{Algorithm}
\newtheorem*{theorem*}{Theorem}
\newtheorem*{corollary*}{Corollary}
\newtheorem*{proposition*}{Proposition}

\newcommand{\Z}{\mathbb{Z}}
\newcommand{\R}{\mathbb{R}}

\newcommand{\Pc}{\mathcal{P}}

\newcommand{\polytope}{P}
\newcommand{\poset}{\Pi}

\def\Rc{\mathcal{R}}
\newcommand{\Ext}{\operatorname{Ext}}

\newcommand{\hilb}{h}

\DeclareMathOperator{\coneZOperator}{\operatorname{C_{\Z}}}
\newcommand{\coneZ}[1]{\coneZOperator(#1)}
\newcommand{\Hilb}{H}
\newcommand{\semigroup}{\mathcal{S}}

\newcommand{\addTop}[1]{\overline{#1}}
\newcommand{\addBottom}[1]{\underline{#1}}
\newcommand{\posetTop}{\addTop{\poset}}
\newcommand{\posetBottom}{\addBottom{\poset}}
\newcommand{\posetTopBottom}{\addBottom{\posetTop}}
\DeclareMathOperator{\graphOperator}{\Gamma}
\newcommand{\graph}[1]{\graphOperator(#1)}
\DeclareMathOperator{\coneOperator}{cone}

\newcommand{\order}{\mathcal{O}}
\newcommand{\chain}{\mathcal{C}}
\newcommand{\cone}{\operatorname{cone}}
\newcommand{\coneOPi}{\coneOperator(\order(\poset))}
\newcommand{\semigroupOperator}{\mathcal{S}}
\newcommand{\semigroupOPi}{\semigroupOperator(\order(\poset))}
\DeclareMathOperator{\indicatorOperator}{\mathbbm{1}}
\newcommand{\indicator}[1]{\indicatorOperator_{#1}}
\DeclareMathOperator{\canonicalOperator}{\omega}
\newcommand{\canonical}[1]{\canonicalOperator_{#1}}
\newcommand{\Ca}{\operatorname{Cayley}}
\newcommand{\codeg}{\operatorname{codeg}}
\newcommand{\conv}{\operatorname{conv}}
\let\int\relax
\DeclareMathOperator{\int}{int}

\DeclareMathOperator{\lengthOperator}{length}
\newcommand{\length}[1]{\lengthOperator(#1)}
\newcommand{\height}{\operatorname{height}}
\newcommand{\depth}{\operatorname{depth}}
\newcommand{\rk}{\operatorname{rank}}
\newcommand{\path}{\mathcal{P}}

\newcommand{\x}{\boldsymbol{x}}
\newcommand{\y}{\boldsymbol{y}}
\newcommand{\p}{\boldsymbol{p}}
\newcommand{\z}{\boldsymbol{z}}
\renewcommand{\k}{\Bbbk}
\renewcommand{\v}{\boldsymbol{v}}
\renewcommand{\u}{\boldsymbol{u}}

\renewcommand{\b}{\mathbf{b}}

\newcommand{\flocomment}[1]{ {\textcolor{violet}{ #1  --Florian}}} 

\newcommand{\christiancomment}[1]{ {\textcolor{purple}{ #1  --Christian}}}  
\newcommand{\rem}[1]{} 
\newcommand{\ehr}{\operatorname{ehr}}
\newcommand{\Ehr}{\operatorname{Ehr}} 
\numberwithin{equation}{section}
\title{Levelness of Order Polytopes}
	
	\author{Christian Haase}
\address{Fachbereich Mathematik und Informatik, Freie Universit\"at Berlin, Germany}	
	\email{haase@math.fu-berlin.de}
	
	\author{Florian Kohl}

	\address{Fachbereich Mathematik und Informatik, Freie Universit\"at Berlin, Germany}
\email{fkohl@math.fu-berlin.de}	
	\author{Akiyoshi Tsuchiya}
	
	\address{Department of Pure and Applied Mathematics, Graduate School of Information Science and Technology, Osaka University, Suita, Osaka 565-0871, Japan}

\email{a-tsuchiya@ist.osaka-u.ac.jp}

\begin{document}

\maketitle

	\begin{abstract}
		Since their introduction by Stanley~\cite{StanleyOrderPoly} order polytopes have been intriguing mathematicians as their geometry can be used to examine (algebraic) properties of finite posets. In this paper, we follow this route to examine the levelness property of order polytopes. The levelness property was also introduced by Stanley~\cite{Stanley-CM-complexes} and it generalizes the Gorenstein property. This property has been recently characterized by Miyazaki~\cite{Miyazaki} for the case of order polytopes. We provide an alternative characterization using weighted digraphs. Using this characterization, we give a new infinite family of level posets and show that determining levelness is in $\operatorname{co-NP}$. This family can be used to create infinitely many examples illustrating that the levelness property can not be characterized by the $h^{\ast}$-vector. We then turn to the more general family of alcoved polytopes. We give a characterization for levelness of alcoved polytopes using the Minkowski sum. Then we study several cases when the product of two polytopes is level. In particular, we provide an example where the product of two level polytopes is not level. 
		
	\end{abstract}
	
	\section{Introduction}
	\label{sec:Intro}
	Partially ordered sets --- or posets for short --- are ubiquitous objects in mathematics. One particularly nice way to study them was introduced by Richard Stanley. In \cite{StanleyOrderPoly}, he associated two geometric objects to every finite poset $\poset$, namely the order polytope $\order(\poset)$ and the chain polytope $\chain(\poset)$. These objects encode important information about the underlying poset $\poset$. For instance, the vertices of the order polytope are given by the indicator vectors of order filters, hence showing that the number of order filters equals the number of vertices of $\order(\poset)$. Furthermore, for a poset $\poset$ on $d$ elements, the Ehrhart polynomial
	\[
	\ehr_{\order(\poset)} (t) : = \# t \order(\poset) \cap \Z^d,
	\]
	which counts the number of integer points in dilates of $\order(\poset)$, equals the number of order-preserving maps $\poset \rightarrow [t+1]: = \{1,2,\dots, t+1 \}$, see \cite[Thm. 4.1]{StanleyOrderPoly}. It is often more convenient to consider the formal power series
	\[
	\Ehr_{\order(\poset)} (z) = 1 + \sum_{k \geq 1} \ehr_{\order(\poset)} (k) z^k = \frac{h^{\ast}(z)}{(1-z)^{d+1}} ,
	\]
	which is called the Ehrhart series of $\order(\poset)$. Finding and understanding the coefficients of $h^{\ast}$-polynomial is at the heart of Ehrhart theory. Therefore, inequalities for the coefficients are of special interest, see \cite{Stapledon1, Stapledon2, HibiInequalities, StanleyInequalities}. Hofscheier, Katth\"an, and Nill proved a structural result about $h^{\ast}$-polynomials, see \cite[Thm. 3.1]{Spanning}, where they showed that if the integer points of a lattice polytope span the integer lattice, then $h^{\ast}$ cannot have internal zeros. There are even some universal inequalities for $h^{\ast}$-vectors, see \cite{BallettiHigashitani}.  
	
	In this paper, we are particularly interested in the levelness property of order polytopes, where we recall that a polytope is level if its canonical module is generated by elements of the same degree, see Definition~\ref{def:level}. Levelness was first introduced by Stanley~\cite[p.~54]{Stanley-CM-complexes} and it generalizes the Gorenstein property. While the Gorenstein property can be completely classified by the $h^{\ast}$-vector being palindromic \cite{Stanley-HilbertFunctions}, the same is not true for the levelness property, see also Remark~\ref{rem:HibiExampleSameHStar}. However, we have the following inequalities:
	\begin{proposition}[{\cite[3.3 Prop]{StanleyGreenBook}}]
		Let $\mathcal{R}=\bigoplus_{i\in\Z_{\geq 0}} \mathcal{R}_i$ be a standard level $\k$-algebra with Hilbert series
		\[
		\operatorname{Hilb}(\mathcal{R}, t) = \frac{h^{\ast}_0+ h_1^{\ast}t + \dots + h^{\ast}_s t^s}{(1-t)^d},
		\]
		where $h^{\ast}_s \neq 0$. Then, for all $i$, $j$ with $h_{i+j}^{\ast}>0$, we have $h_i^{\ast} \leq h_j^{\ast}h_{i+j}^{\ast}$.
	\end{proposition}
	In Section~\ref{sec:Background}, we will see that this result implies non-trivial inequalities for the $h^{\ast}$-polynomial of order polytopes that are not satisfied by all lattice polytopes.  
	
	While the Gorenstein property has been studied extensively in the past decades, the levelness property has only fairly recently been examined for certain classes of polytopes, with the exception of \cite{HibiStraightening}. Recent examples include \cite{EneHerzogHibi,HigashitaniYanagawa-Level,KohlOlsen}. In this paper, we focus on the levelness property of order polytopes, i. e., the levelness property of the Ehrhart ring of the order polytope.  Hibi~\cite{HibiGorenstein} was the first to examine minimal elements of the canonical module of this Ehrhart ring, which in this context is also known as the Hibi ring. In particular, he characterized Gorenstein posets. The biggest influence on this paper comes from \cite{Miyazaki}, where Miyazaki examines and characterizes levelness of order polytopes. We provide an alternative characterization using weighted digraphs $\Gamma(\poset')$ coming from subposets $\poset' \subset \poset \cup \{ \pm \infty\}$, see Definition~\ref{def:WeightedGraphs} for details.
	\begin{corollary*}[see Corollary~\ref{cor:GeneralGamma'}]
		Let $\poset$ be a finite poset. $\poset$ is level if and only if for all $\Gamma(\poset')$ that do not have a negative cycle, the digraph $\Gamma(\poset' \cup \{\textnormal{longest chains in } \posetTopBottom\})$ does not have a negative cycle.
	\end{corollary*}
	This corollary enables us to use the Bellman--Ford algorithm to check levelness. As a direct consequence, we get that determining levelness is in $\operatorname{co-NP}$:
	\begin{corollary*}[see Corollary~\ref{cor:co-NP}]
		Levelness of order polytopes is in $\operatorname{co-NP}$.
	\end{corollary*}
	
	We show that the necessary condition for levelness of order polytopes in \cite[Thm. 4.1]{EneHerzogHibi} is indeed equivalent to a special case of our characterization. Furthermore, we give an example that was related to us by Alex Fink showing that this condition is not sufficient, see Remark~\ref{rem:Fink} and Figure~\ref{fig:Fink}.
	\begin{theorem*}[see Theorem~\ref{thm:HibiConditionIsEquivalent}]
		Let $\poset$ be a finite poset and $r= \operatorname{codeg}(\order(\poset))$.
		The following are equivalent:
		\begin{enumerate}
			\item 	The inequality
			\begin{equation*}
				\label{eq:heightdepth}
				\operatorname{height}(j) + \operatorname{depth} (i) \leq \operatorname{rank} (\posetTopBottom) +1
			\end{equation*}
			is satisfied for all $j \gtrdot i \in \poset$.
			\item for all Hasse edges $j\gtrdot i \in \poset$ there is an integer point $\x \in r\int(\order(\poset))$ such that  $x_j = x_i +1$.
		\end{enumerate}
	\end{theorem*}
	
	In Section~\ref{sec:SeriesParallelPosets}, we use Corollary~\ref{cor:GeneralGamma'} to describe an infinite family of level order polytopes. The main ingredient is the ordinal sum of two posets, denoted $\triangleleft$.
	\begin{theorem*}[see Theorem~\ref{thm:OrdSumLevel}]
		The ordinal sum $\poset ={\poset_1} \triangleleft \poset_2$ of two posets $\poset_1$, $\poset_2$ is level if and only if both $\poset_1$ and $\poset_2$ are level.
	\end{theorem*}
	
	Moreover, the ordinal sum operation interacts nicely with the $h^{\ast}$-polynomial.
	\begin{proposition*}[see Proposition~\ref{prop:HStarOrdinalSum}]
		Let $\poset, {\poset_1}, \poset_2$, be finite posets where $\poset := {\poset_1} \triangleleft \poset_2$. Moreover, let $h^{\ast}_{\poset}, h^{\ast}_{\poset_1}, h^{\ast}_{\poset_2}$ be the $h^{\ast}$-polynomial of the Ehrhart series of the corresponding order polytopes. Then
		\begin{equation*}
			\label{eq:hstarordsum}
			h^{\ast}_{\poset} =  h^{\ast}_{\poset_1}  h^{\ast}_{\poset_2}.
		\end{equation*} 
	\end{proposition*}
	
	As we illustrate in Remark~\ref{rem:HibiExampleSameHStar}, Theorem~\ref{thm:OrdSumLevel} and Proposition~\ref{prop:HStarOrdinalSum} together can be used to create infinitely many examples of pairs of posets that have the same Ehrhart polynomial, but where one poset is level and the other one is not.
	
	We then turn to the more general class of alcoved polytopes. We give a Minkowski sum characterization of levelness for alcoved polytopes.
	\begin{proposition*}[see Proposition~\ref{al_level}]
		Let $\polytope \subset \R^d$ be an alcoved polytope and let $r=\textnormal{codeg}(\polytope)$.
		Then $\polytope$ is level if and only if
		for any integer $k \geq r$,
		it follows that  $(k\polytope)^{(1)}=(r\polytope)^{(1)}+(k-r)\polytope$, where $(l\polytope)^{(1)} : = \conv(l\polytope^{\circ} \cap \Z^d)$ for $l \in \{r,k \}$. 
	\end{proposition*}	
	
	Then we examine when the Cartesian product of two alcoved polytopes is level. We arrive at the following result:
	\begin{theorem*}[see Theorem~	\ref{thm:prod}]
		Let $\polytope \subset \R^d$ and $Q \subset \R^{e}$ be alcoved polytopes.
		Suppose that $Q$ is level and  $r=\textnormal{codeg}(Q) \geq \dim \polytope+1$.
		Then $P \times Q \subset \R^{d+e}$ is level.
	\end{theorem*}
	This results shows that --- under the right assumptions --- the product of a level polytope with a non-level polytope can indeed be guaranteed to be level.
	\begin{theorem*}[see Theorem~\ref{thm:prod2}]
		Let $\poset$ be a poset on $d$ elements and $\poset_1,\ldots,\poset_m$ the connected components of $\Pi$.
		If each $\poset_i$ is level, then $\poset$ is level.
	\end{theorem*}
	This theorem tells us that in order to guarantee levelness of a poset it is sufficient to show that all components are level. More generally, for Cartesian products of level polytopes we have:
	\begin{theorem*}[see Theorem~\ref{thm:ProdOfLevel}]
		Let $P \subset \R^d$ and $Q\subset \R^e$ be level polytopes. If either
		\begin{enumerate}
			\item $\codeg(Q)< \codeg (P)$ and $Q$ has the integer-decomposition property,
			\item $\codeg(P)< \codeg (Q)$ and $P$ has the integer-decomposition property,
			\item or if $\codeg(Q) = \codeg (P)$,
		\end{enumerate}
		then $P \times Q$ is level. 
	\end{theorem*}
	The assumptions are indeed necessary. In Remark~\ref{rem:ProdOfLevel}, we give an explicit example of two level polytopes whose product is not level.

	The structure of this paper is as follows: In Section~\ref{sec:Background}, we recall the basics of Ehrhart theory relevant to this paper, we introduce posets, order polytopes, and chain polytopes, and we show how these relate to combinatorial, commutative algebra. In Section~\ref{sec:Miyazaki}, we recall Miyazaki's results on level posets. We then give an alternative characterization of level posets using weighted digraphs. This is done in Section~\ref{sec:NewAlgo}. In Section~\ref{sec:NecessaryCondition}, we show that this characterization generalizes a necessary condition of Ene, Herzog, Hibi, and Saeedi. In Section~\ref{sec:SeriesParallelPosets}, we use this characterization to examine levelness of series-parallel posets. In the last section, we examine levelness of alcoved polytopes and examine when certain products of polytopes are again level.

	\rem{
		level algebras by Stanley~\cite[p.~54]{Stanley-CM-complexes}
		generalizing Gorenstein -- very subtle (e.g., not just Hilbert~\cite{HibiStraightening}), desirable enumerative
		consequences
		
		levelness property of the semigroup algebras of order polytopes 
		has recently come into focus~\cite{EneHerzogHibi} \cite{Miyazaki}

		mention stanley's work, hibi herzog ene et al
		miyazaki}
	\section{Acknowledgements}
	The authors would like to thank Francisco Santos for valuable feedback and constructive criticism while reading a draft of this article.	The second author was supported by a scholarship of the Berlin Mathematical School. He would also like to explicitly thank Takayuki Hibi and the third author for organizing the Workshop on Convex Polytopes for Graduate Students at Osaka University in January 2017, which is where this work began. 
	\section{Background and Notation}
	\label{sec:Background}
	\subsection{Ehrhart Theory and Lattice Polytopes}
	
	Polytopes correspond to solutions of finitely many linear inequalities, where the set of solutions is bounded. Integral solutions to these inequalities correspond to integer points in polytopes. Counting the number of these solutions leads directly to Ehrhart theory, which is the study of (the number of) integer points in a given polytope. In this subsection, we  give a brief and incomplete introduction to this beautiful area. We refer the interested reader to the excellent books~\cite{CCD,Ziegler}.
	
	A \emph{polytope} $\polytope \subset \R^d$ is the convex hull of finitely many points $\u_1$, $\u_2$, $\dots$, $\u_r \in \R^d$, i. e., 
	\[
	\polytope = \conv \{\u_1,\u_2,\dots, \u_r \} := \left\{\sum_{i=1}^r \lambda_i \u_i \colon \, \sum_{i=1}^r\lambda_i = 1, \,\lambda_i \geq 0 \text{ for all }i  \right\}.
	\]
	The inclusion-minimal set $\{\v_1,\v_2,\dots,\v_s \}$ such that $\polytope = \conv\{\v_1,\dots,\v_s \}$ is called the \emph{vertex set of $\polytope$} and its elements are called the \emph{vertices}. A polytope whose vertex set is contained in $\Z^d$, is called a \emph{lattice polytope}. We define the dimension of a polytope to be the dimension $n$ of its affine span. We then call $\polytope$  an $n$-polytope. We remark that some authors reserve the term lattice polytope for polytopes whose vertices are contained in a general lattice $\Lambda$. However, we will only consider the case where $\Lambda = \Z^e$ for appropriate $e$. 
	\begin{definition}
		\label{def:Ehrhart polynomial}
		Let $\polytope \subset \R^d$ be a $d$-polytope. We define the \emph{Ehrhart function $\ehr_{\polytope} \colon \Z_{\geq 0} \rightarrow \Z_{\geq 0}$}
		\[
		\ehr_{\polytope} (k) = \# k\polytope\cap \Z^d.
		\] 
		When $P$ is a lattice polytope, we call $\ehr_{\polytope}$ the \emph{Ehrhart polynomial of $\polytope$}. The \emph{Ehrhart series} of a lattice polytope $\polytope$ is the formal power series
		\[
		\Ehr_{\polytope} (z) = 1 + \sum_{k \geq 1} \ehr_{\polytope} (k) z^k .
		\]
	\end{definition}
	
	Ehrhart~\cite{Ehrhart} famously proved that the Ehrhart function is a quasipolynomial if the vertices of ${\polytope}$ are rational and it is even a polynomial if the vertices are lattice points. As a direct consequence, the Ehrhart series can be written as a rational function
	\[
	\Ehr_P (z)=\frac{h^{\ast}_0 + \dots+h^{\ast}_d z^d}{(1-z)^{d+1}}
	\]
	and the numerator of this rational function is called the \emph{$h^{\ast}$-polynomial of ${\polytope}$}.
	We can actually say more about the coefficients of the numerator polynomial, see~\cite{StanleyNonnegativity}:
	\begin{theorem}[Stanley's nonnegativity theorem]
		\label{thm:Stanleynonnegativity}
		Let ${\polytope}$ be a lattice $d$-polytope. Then
		\[
		\Ehr_{\polytope} (z) = \frac{h^{\ast}_0 + \dots+h^{\ast}_d z^d}{(1-z)^{d+1}},
		\]
		where $h^{\ast}_0, h^{\ast}_1, \dots, h^{\ast}_d \in \Z_{\geq 0}$.
	\end{theorem}
	The non-negativity of the coefficients of the $h^{\ast}$-polynomial is deeply related to what is called the Cohen--Macaulay property. We now define the \emph{degree} and \emph{codegree} of a lattice polytope.
	\begin{definition}
		\label{def:DegAndCodeg}
		Let ${\polytope}$ be a lattice $d$-polytope with $h^{\ast}$-polynomial $ h^{\ast}_0 + \dots+h^{\ast}_d z^d$. The \emph{degree of ${\polytope}$} is defined as
		\[
		\operatorname{deg} ({\polytope}) = \max \{k \colon h^{\ast}_k \neq 0 \}.
		\]
		The \emph{codegree of ${\polytope}$} is defined as
		\[
		\codeg(P) = d+1 - \operatorname{deg}({\polytope}).
		\]
	\end{definition}
	
	\begin{remark}
		\label{rem:CodegreeInterpretation}
		The codegree of a lattice $d$-polytope ${\polytope}$ is the smallest positive integer $c$ such that $c{\polytope}$ contains an interior integer point. This is a consequence of a reciprocity theorem relating the Hilbert series of $\k[\polytope]$ to the Hilbert series of the canonical module $\omega_{\k[\polytope]}$, see for instance \cite[Thm. 6.41]{BrunsGubeladze}.
	\end{remark}

	\subsection{Two Poset Polytopes}
	A \emph{\textbf{p}artially \textbf{o}rdered \textbf{set} }(or poset) $(\poset, \leq_{\poset})$ is a set $\poset$ together with a binary relation
	$\leq_{\poset}$ that is reflexive, antisymmetric, and transitive. The
	relation $\leq_{\poset}$ is called a \emph{partial order} and when there
	is no confusion about the poset, we simply write $\leq$. An element $j
	\in \poset$ is said to \emph{cover} an element $i \in \poset$, denoted $j\gtrdot i$, if $i \leq k \leq j $ implies that either $i=k$ or $j=k$. One can recover all partial orders from these cover relations. Therefore, it's convenient to illustrate the poset using these cover relations by a \emph{Hasse diagram}, see Figure~\ref{fig:HasseDiag}.
	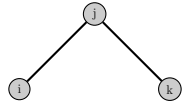
\begin{figure}[h]
		\center
		\begin{tikzpicture}[darkstyle/.style={circle,draw,fill=gray!40,minimum size=20}, scale = .5]
		\node[circle,draw=black,fill=white!80!black,minimum size=20, scale=.4] (1) at (4,0) {k };
		\node[circle,draw=black,fill=white!80!black,minimum size=20, scale=.4] (2) at (0,0) { i};
		\node[circle,draw=black,fill=white!80!black,minimum size=20, scale=.4] (3) at (2,2) { j};
		\draw[thick] (1)--(3)--(2);
		\end{tikzpicture}
		\caption{The Hasse diagram of the poset $i \lessdot j \gtrdot k$.}
		\label{fig:HasseDiag}
	\end{figure}
	
	Given a poset $\poset$, we define the poset $\posetTop = (\poset \cup \{\infty \}, \leq_{\posetTop})$, where 
	\[
	i <_{\posetTop} j :\Longleftrightarrow
	\begin{cases}
	j = \infty \text{ and } i\in \Pi, \\
	i <_{\poset} j.
	\end{cases}
	\]
	Similarly, we define $\posetBottom = (\poset \cup \{-\infty \}, \leq_{\posetBottom})$, where
	\[
	i <_{\posetBottom} j:\Longleftrightarrow
	\begin{cases}
	i = -\infty \text{ and } j\in \Pi, \\
	i <_{\poset}j .
	\end{cases}
	\]

	To every finite poset, Stanley associated two geometric objects,
	namely the order polytope and the chain polytope:
	\begin{definition}[\cite{StanleyOrderPoly}, Def. 1.1]
		\label{def:orderpolytope}
		The \emph{order polytope $\order(\poset)$} of a finite poset $\poset$
		is the subset of $\R^{\poset} = \{f\colon \poset \rightarrow \R \}$ defined
		by
		\begin{align*}
			&0 \leq f(i) \leq 1  &\qquad \text{for all } i\in \poset, \\
			&f(i)\leq f(j)      &\qquad \text{if } i\leq_{\poset} j.\\
		\end{align*}
	\end{definition}
	\begin{definition}[\cite{StanleyOrderPoly}, Def. 2.1]
		\label{def:chainpolytope}
		The \emph{chain polytope $\chain(\poset)$} of a finite poset $\poset$ is the subset of $\R^{\poset} = \{g\colon \poset \rightarrow \R \}$ defined by the conditions
		\begin{align*}
			&0 \leq g(i)   &\qquad \text{for all } i\in \poset, \\
			&g(i_1) + g(i_2)+ \dots g(i_k)\leq 1      &\qquad \text{for all chains } i_1 <_{\poset} i_2 <_{\poset} \dots <_{\poset} i_k \text{ of }\poset.\\
		\end{align*}
	\end{definition}
	
	\begin{remark}
		In the following, we will use an isomorphism $\R^{\poset} \cong \R^{\# \poset}$ to make notation better.
	\end{remark}
	\begin{figure}[h]
		\centering
		\begin{subfigure}{.5\textwidth}
			\centering
			\includegraphics[width=.4\linewidth]{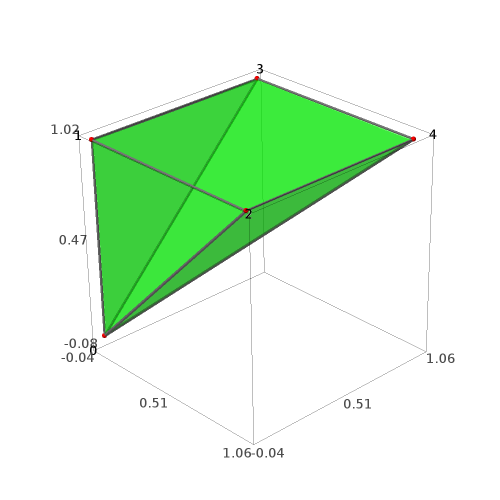} 
			
			\label{fig:OrderPoly}
		\end{subfigure}%
		\begin{subfigure}{.5\textwidth}
			\centering
			\includegraphics[width=.4\linewidth]{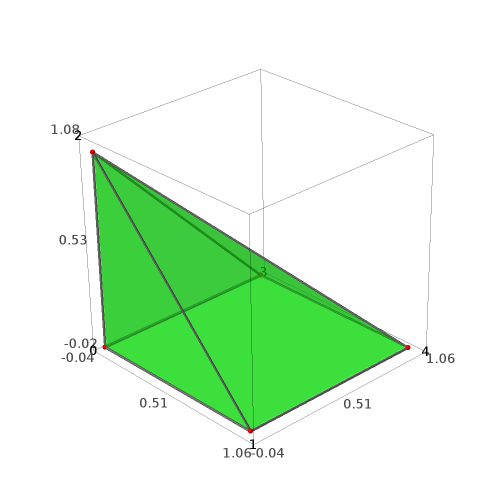} 
			
			\label{fig:ChainPoly}
		\end{subfigure}
		\caption{The order and chain polytope of the poset described in Figure~\ref{fig:HasseDiag}.}
		\label{fig:OrderAndChainPoly}
	\end{figure}
	
	We define an \emph{order filter $F$ } of a poset $\poset$ to be a subset $F\subset \poset$ such that if $i \in F$ and $i<j$, then $j\in F$. To every filter $F$, one can associate a \emph{characteristic function} $\indicator F$ defined as
	\[
	\indicator F (i) : = \begin{cases} 1 &\text{ if } i\in F,\\
	0 &\text{ otherwise.}\\
	
	\end{cases}
	\]
	Stanley showed that vertices of $\order(\poset)$ are given by the characteristic functions of order filters.
	\begin{corollary}[{\cite[Cor. 1.3]{StanleyOrderPoly}}]
		\label{cor:VerticesOrderPoly}
		The vertices of $\order(\poset)$ are the characteristic functions  $\indicator F $ of order filters $F$. In particular, the number of vertices equals the number of order filters.
	\end{corollary}
	
	Stanley also gave the vertex description of chain polytopes. We define an \emph{antichain $A$} of a poset $\poset$ to be a subset $A \subset \poset$ of pairwise incomparable elements. The characteristic function $\indicator A$ of an antichain $A$ is defined similarly to the characteristic function of an order filter.
	\begin{theorem}[{\cite[Thm 2.2]{StanleyOrderPoly}}]
		\label{thm:VerticesChainPoly}
		The vertices of $\chain(\poset)$ are given by the characteristic functions $\indicator A$ of antichains $A$. In particular, the number of vertices of $\chain(\poset)$ equals the number of antichains of $\poset$.
	\end{theorem}
	
	Let $\poset$ be a $d$-element poset and let $m\in \Z_{\geq 1}$. We define $\Omega (\poset,m)$ to be the number of \emph{order-preserving maps} $\poset \to \{1,2,\dots,m \}$, where we say that a map $f$ is \emph{order preserving} if $i \leq_{\poset} j$ implies $f(i) \leq f(j)$. These order-preserving maps correspond to integer points in dilates of the order polytope as the next theorem shows:
	\begin{theorem}[{\cite[Thm. 4.1]{StanleyOrderPoly}}]
		\label{thm:EhrhartPolyofOrder}
		The Ehrhart polynomials of $\order(\poset)$ and $\chain (\poset)$ are given by
		\[
		\ehr_{\order(\poset)}(k) = \ehr_{\chain(\poset)}(k) = \Omega(\poset, k+1).
		\]
	\end{theorem}
	\begin{remark}
		\label{rem:InteriorLatticePointsStrictlyOrderPreservingMaps}
		As is implicit in Stanley's proof, interior integer points are in bijection with \emph{strictly order-preserving  maps}, i.e., maps $f$ that satisfy $i <_{\poset} j$ implies $f(i)< f(j)$.
	\end{remark}
	In a poset $\poset$, maximal chains can have different lengths. 
	A chain with maximum length is called a \emph{longest chain}. 
	\begin{figure}[h]
		\center
		\begin{tikzpicture}[darkstyle/.style={circle,draw,fill=gray!40,minimum size=20}, scale = .5]
		\node[circle,draw=black,fill=white!80!black,minimum size=20, scale=.4] (1) at (4,0) { };
		\node[circle,draw=black,fill=white!80!black,minimum size=20, scale=.4] (2) at (0,2) { };
		\node[circle,draw=black,fill=white!80!black,minimum size=20, scale=.4] (3) at (0,4) { };
		\node[circle,draw=black,fill=white!80!black,minimum size=20, scale=.4] (4) at (4,6) { };
		\node[circle,draw=black,fill=white!80!black,minimum size=20, scale=.4] (5) at (8,3) { };
		\draw[thick, red] (5) -- (1)--(2)--(3) -- (4)-- (5);
		\draw[thick](5) -- (1);
		\draw[thick](5) -- (4);
		
		\end{tikzpicture}
		\caption{Edges belonging to longest chain are colored red.}
	\end{figure}
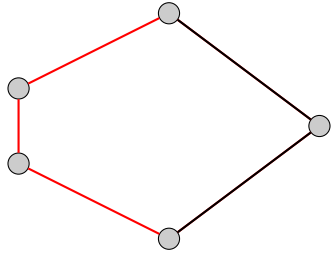
	
	\begin{remark}
		The codegree of $\order(\poset)$ equals the rank of $\posetTopBottom$, i.e., it equals the number of edges in the longest chain of $\posetTopBottom$.
	\end{remark}
	
	\subsection{Level Affine Semigroups}
	This subsection is based on \cite[Ch. 6]{BrunsGubeladze}. Let $\polytope\subset \R^d$ be a lattice $d$-polytope with vertex set $V(\polytope)$ and let $\k$ be an algebraically closed field of characteristic zero. We define the \emph{cone over $\polytope$} as
	\[
	\cone(\polytope): ={\rm span}_{\R_{\geq 0}}\{(\v,1) \, : \, \v\in V(\polytope)\}\subset \R^d\times \R.
	\]
	The set 
	\[
	\semigroup(\polytope): = \left\{ \x \colon \x = \sum_{i \in I} \lambda_i (\v_i,1) \text{, for } \v_i \in \polytope\cap \Z^{d}\text{ and } \lambda_i \in \Z_{\geq 0},\, \forall i \in I  \right\}
	\]
	forms an additive \emph{semigroup}, i.e., a set that is closed under addition, which contains a neutral element, and where addition is associative. We remark that some authors don't require semigroups to contain the neutral element. Moreover, we define the semigroup $\coneZ {\polytope}: = \cone(\polytope)\cap \Z^{d+1}$. We say that $\polytope$ has the \emph{integer-decomposition property} (IDP) if $ \semigroup(\polytope) = \coneZ {\polytope}$. This semigroup gives rise to the \emph{Ehrhart ring of $P$} 
	\[
	\k[P]:=\k[\coneZ{P}]=\k[\x^{\p}\cdot y^m \, : \, (\p,m)\in  \coneZ{P}]\subset \k[x_1^{\pm 1},\ldots, x_n^{\pm 1},y].
	\]	
	The semigroup $\coneZ{P}$ is normal, i.e.,  its semigroup ring $\k[P]$ is normal.
	\begin{remark}
		Since every order polytope has a unimodular triangulation, order polytopes have the integer-decomposition property. Hence, we have
		\[
		\k[P] = \k[\coneZ{P}] = \k[\semigroup(P)].
		\]
	\end{remark}
	The ring $\k[\polytope]$ is also a finitely generated $\k$-algebra of Krull dimension $d+1$, which inherits a natural $\Z_{\geq 0}$-grading given by the $y$-degree. Therefore, we can write it as $\k[\polytope]=\bigoplus_{i\in\Z_{\geq 0}} \k[\polytope]_i$, where $\k[\polytope]_i$ is the $\k$-vector space generated by the degree $i$ monomials. We consider the \emph{Hilbert function} $\hilb \colon \Z_{\geq 0} \to \Z_{\geq 0}$ given by $\hilb (t) = \dim_{\k} \k[\polytope]_t$. As with the Ehrhart function, it makes sense to examine the formal power series
	\[
	\Hilb_{\k[\polytope]}(z) = \sum_{t\geq 0}\hilb(t) z^t ,
	\]
	which we call the \emph{Hilbert series of $\k[\polytope]$.} This Hilbert series is in fact a rational function:
	\begin{theorem}[{\cite[Thm. 6.39, 6.40]{BrunsGubeladze}}]
		\label{thm:ch1HilbertSerre}
		Let $\polytope$ be a lattice $d$-polytope. Then the Hilbert series of $\k[\polytope]$ is of the form
		\begin{equation*}
			\Hilb_{\k[\polytope]}(z) = \frac{n(z)}{(1-z)^{d+1}},
		\end{equation*}
		where $n$ is a polynomial with non-negative, integral coefficients. 
	\end{theorem}
	The reader might have noticed that $\ehr_P (t)= \hilb_P (t)$ and that $\Ehr_P(z) = \Hilb_{\k[\polytope]}(z)$, a fact that is used in \cite[Sec. 12.1]{MillerSturmfels} to prove the polynomiality of $\ehr_P$. 
	
	By a seminal result of Melvin Hochster \cite{Hochster}, $\k[P]$ is Cohen--Macaulay:
	\begin{theorem}[{\cite[Thm. 6.10]{BrunsGubeladze}}]
		\label{thm:ch1Hochster}
		Let $\semigroup$ be a normal, affine semigroup. Then $\k[\semigroup]$ is Cohen--Macaulay for every field $\k$.
	\end{theorem}
	Several important properties of $\k[\polytope]$ can be stated in terms of what is called the \emph{canonical module $\canonical {\k[\polytope]}$}. Sometimes the canonical module is also called the dualizing module.
	
	Let $\mathcal{R}=\bigoplus_{i\in\Z_{\geq 0}} \mathcal{R}_i$ be a finitely generated $\Z_{\geq 0}$-graded $\k$-algebra of Krull dimension $d$. Suppose that $\Rc$ is local and Cohen-Macaulay.
	\begin{definition}
		\label{def:ch1Gorenstein}
		The \emph{canonical module} of $\Rc$, $\omega_{\Rc}$, is the unique module (up to isomorphism)  such that $\Ext_\Rc^d(\k,\omega_{\Rc})=\k$ and $\Ext_{\Rc}^i(\k,\omega_\Rc)=0$ when $i\neq d$. We say that $\Rc$ is \emph{Gorenstein} if $\omega_\Rc\cong\Rc$ as an $\Rc$-module, or equivalently is $\omega_\Rc$ is generated by a single element. A lattice polytope $P$ is Gorenstein if $\k[P]$ is Gorenstein.
	\end{definition}
	However, we will only work with semigroup rings $\k[\polytope]$ arising from lattice polytopes $\polytope$. In this case, the canonical module has a particularly nice description, which is why we omit a proper definition of the $\Ext$-functor:
	\begin{theorem}[Danilov--Stanley {\cite[Thm. 6.31]{BrunsGubeladze}}]
		\label{thm:ch1DanilovStanley}
		Let the notation be as above and let $\k$ be a field. Then the ideal generated by the monomials corresponding to interior integer points $\int(\coneZ{\polytope})$ is the canonical module of $\k[\polytope]$.
	\end{theorem}
	
	We can now state our main definition:
	\begin{definition}
		\label{def:level}
		We say that a lattice polytope $\polytope$ is \emph{level} if the $\k[\polytope]$-module $\omega_{\k[\polytope]}$ is --- as a $\k[\polytope]$-module --- generated by elements of degree $\codeg(\polytope)$. We say that a finite poset $\poset$ is level if $\order(\poset)$ is level.
	\end{definition}
	\begin{remark}
		\label{rem:LevelExplanation}
		We would like to remark that this is equivalent to saying that a lattice $d$-polytope $\polytope$ is level if and only if for all $\x \in \int (\cone(\polytope)) \cap \Z^{d+1}$ there exists a point $\y \in \int (\cone(\polytope)) \cap \Z^{d+1}$ at height $\codeg (\polytope)$ such that $\x - \y \in \cone(\polytope)$.
	\end{remark}
	\begin{definition}
		\label{def:MinimalLatticePoints}
		We say that $\x \in \int (\cone(\polytope)) \cap \Z^{d+1}$ is \emph{minimal} if the corresponding monomial is a minimal generator of $\omega_{\k[\polytope]}$.
	\end{definition}
	\begin{example}
		Let $P: =  [0,2] \times [0,1]$. Then the cone over the polytope is 
		\[
		\cone(P) =  \{\x = \lambda_1 (0,0,1) + \lambda_2 (0,1,1)+\lambda_3 (2,0,1)+\lambda_4 (2,1,1) \colon \lambda_i \geq 0 \}.
		\]
		
		\begin{figure}[h]
			\centering
			\tdplotsetmaincoords{70}{120}
			\begin{tikzpicture}
			[tdplot_main_coords,
			grid/.style={dotted,black},
			axis/.style={->,black,thick},
			my fill/.style={fill=gray, fill opacity=.2},
			myfill/.style={fill=gray, fill opacity=.2},
			mf/.style={fill=blue, fill opacity=.2},
			cube/.style={opacity=.3,very thick,fill=white},scale=1]
			
			\foreach \z in {1,2}
			\foreach \x in {0,.5,...,\z}
			\foreach \y in {0,1,...,\z}
			\pgfmathtruncatemacro{\cur}{(2 * \x}
			\node[circle,draw=black,fill=white!80!black,minimum size=20, scale=.1] at (\cur,\y,\z) {};
			
			\node[circle,draw=black,fill=black,minimum size=20, scale=.1] at (1,0,1) {};
			\node[circle,draw=black,fill=white!80!black,minimum size=20, scale=.1] at (0,0,1) {};
			\node[circle,draw=black,fill=white!80!black,minimum size=20, scale=.1] at (0,1,1) {};
			\node[circle,draw=black,fill=white!80!black,minimum size=20, scale=.1] at (1,1,1) {};
			\node[circle,draw=black,fill=red,minimum size=30, scale=.1] at (1,1,2) {};
			\node[circle,draw=black,fill=red,minimum size=30, scale=.1] at (3,1,2) {};

			\node[circle,draw=black,fill=red,minimum size=30, scale=.1] at (2,1,2) {};
			
			\node[circle,draw=black,fill=white!80!black,minimum size=20, scale=.1] at (2,0,2) {};
			\node[circle,draw=black,fill=white!80!black,minimum size=20, scale=.1] at (2,2,2) {};
			\node[circle,draw=black,fill=white!80!black,minimum size=20, scale=.1] at (0,2,2) {};
			\node[circle,draw=black,fill=white!80!black,minimum size=20, scale=.1] at (0,0,2) {};
			\node[circle,draw=black,fill=white!80!black,minimum size=20, scale=.1] at (1,0,2) {};
			\node[circle,draw=black,fill=white!80!black,minimum size=20, scale=.1] at (0,1,2) {};
			\node[circle,draw=black,fill=white!80!black,minimum size=20, scale=.1] at (1,2,2) {};
			\draw[axis] (0,0,0) -- (7,0,0) node[anchor=west]{$x_1$};
			\draw[axis] (0,0,0) -- (0,7,0) node[anchor=west]{$x_2$};
			\draw[axis] (0,0,0) -- (0,0,7) node[anchor=west]{$x_3$};
			
			\draw[cube, green] (0,0,1) -- (0,1,1) -- (2,1,1) -- (2,0,1) -- cycle;
			\node[right] at (1.5,1.5,1) {$P$}; 
			\draw[cube, green] (0,0,2) -- (0,2,2) -- (4,2,2) -- (4,0,2) -- cycle;
			\draw[axis,black] (0,0,0) -- (0,0,7);
			\draw[axis,black] (0,0,0) -- (14,0,7);
			\draw[axis,black] (0,0,0) -- (14,7,7);
			\draw[axis,black] (0,0,0) -- (0,7,7);
			
			\draw[axis,blue] (1,1,2) -- (1,1,7);
			\draw[axis,blue] (1,1,2) -- (1,6,7);
			\draw[axis,blue] (3,1,2) -- (13,1,7);
			\draw[axis,blue] (3,1,2) -- (13,6,7);

			\draw[my fill](14,0,7) --  (0,0,0)  -- (14,7,7);
			\draw[my fill](14,7,7) --  (0,0,0)  -- (0,7,7);
			\draw[myfill](0,7,7) --  (0,0,0)  -- (0,0,7);
			\draw[myfill](0,0,7) --  (0,0,0)  -- (14,0,7);

			\draw[mf,blue] (1,6,7) -- (1,1,2)  -- (1,1,7);
			\draw[mf,blue](13,1,7) -- (3,1,2)  -- (13,6,7);
			\draw[mf,blue](13,6,7) --  (3,1,2) -- (1,1,2)  --(1,6,7);
			\draw[mf,blue](1,1,7) -- (1,1,2)-- (3,1,2)  -- (13,1,7);

			\end{tikzpicture}
			\caption{The box $P$ and its dilate $2P$, the cone over $P$ (gray) and the (conical hull of the) canonical module (blue).}
			\label{fig:ch1ExampleSquare}
		\end{figure}
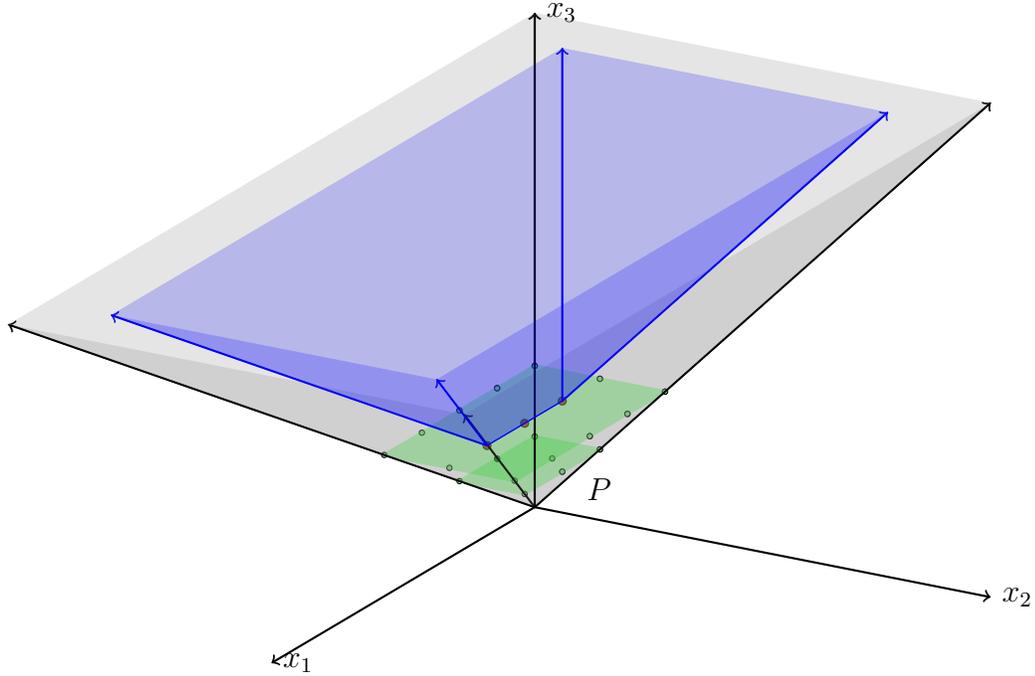
		The semigroup is generated (over $\Z$) by the points $(0,0,1)$, $(2,0,1)$, $(0,1,1)$, and $(2,1,1)$. There are exactly $(t+1)(2t+1)$ many monomials of degree $t$ and thus $\hilb (t) = \dim \k[P]_t = (t+1)(2t+1)$. Figure~\ref{fig:ch1ExampleSquare} shows the monomials of degree less than three. Hence, the Ehrhart series $\Ehr_{P}({z})$ equals
		\[
		\Ehr_{P}({z}) = 1 + \sum_{ k \geq 1} (k+1)(2k+1) z^k  = \frac{1 +  3z}{(1-z)^3} .
		\]
		Thus $P$ is a polytope of degree $1$ with codegree $2$. The canonical module is generated by the three interior integer points of lowest degree (marked red), namely $(1,1,2)$, $(2,1,2)$, and $(3,1,2)$. This shows that $P$ is level. 	
	\end{example}
	
	\section{Miyazaki's characterization}
	\label{sec:Miyazaki}
In this section, we recall the characterization for levelness of order polytopes which was introduced by Miyazaki, see~\cite{Miyazaki}. 
	In order to give a characterization of the level property, Miyazaki defined sequences with condition $N$.
	\begin{definition}[{\cite[Def. 3.1]{Miyazaki}}]
		\label{def:Nsequences}
		Let  $i_1,j_1,i_2,j_2,\ldots,i_t,j_t$,  be a possible empty sequence of elements in a finite poset $\poset$. We say the sequence satisfies condition $N$ if
		\begin{enumerate}
			\item $i_1<j_1>i_2<j_2>\dots>i_t<j_t$ and
			\item for any $m, n$ with $1 \leq m<n\leq t$, $i_m\nleq j_n$.
		\end{enumerate}
	\end{definition}
	\begin{definition}
		Let $i_1,j_1,i_2,j_2,\ldots,i_t,j_t$ be a sequence of elements in a finite poset $\poset$ with condition $N$, and set $j_0=\infty$ and $i_{t+1}=-\infty$.
		We set
		\[
		r(i_1,j_1,\ldots,i_t,j_t):=\sum_{s=1}^{t}(\rk[i_{s},j_{s-1}]-\rk[i_s,j_s])+\rk[i_{t+1},j_t].
		\]
		Moreover, set
		\[r_{\max}:=\max\{ r(i_1,j_1,\ldots,i_t,j_t) : i_1,j_1,\ldots,i_t,j_t \textnormal{ is a sequence with condition } N  \}.
		\]
	\end{definition}
	Associated to every sequence with condition $N$, Miyazaki defines a special element of $\cone(\order(\poset))$:
	\begin{definition}[{\cite[Def 3.6]{Miyazaki}}]
		Let $i_1,j_1,i_2,j_2,\ldots,i_t,j_t$ be a sequence of elements in a finite poset $\poset$ with condition $N$, and set $j_0=\infty$ and $i_{t+1}=-\infty$.
		We define
		\[
		x{(i_1,j_1,\ldots,i_t,j_t)}_{i_m}:=\sum_{s=m}^{t}(\rk[i_{s+1},j_{s}]-\rk[i_s,j_s])
		\]
		for $1 \leq m \leq t+1$ and
		\[
		y{(i_1,j_1,\ldots,i_t,j_t)}_{k}:=\max\{\rk[i_s,k]+x{(i_1,j_1,\ldots,i_t,j_t)}_{i_s} : k \geq i_s\}
		\]
		for $k \in \posetTopBottom$.
	\end{definition} 
	These elements give rise to an important class of minimal elements, as the next lemma shows.
	\begin{lemma}[{\cite[Lem. 3.8]{Miyazaki}}]
		\label{lem:miya}
		Let $i_1,j_1,i_2,j_2,\ldots,i_t,j_t$ be a sequence of elements in a finite poset $\poset$ with condition $N$, and set $j_0=\infty$ and $i_{t+1}=-\infty$.
		If $r(i_1,j_1,\ldots,i_t,j_t)=r_{\max}$, then 
		the element $\y (i_1,j_1,\ldots,i_t,j_t)$ is minimal in the sense of Definition~\ref{def:MinimalLatticePoints}. In particular, it is an interior lattice point in the cone.
		Furthermore, 
		\[
		y(i_1,j_1,\ldots,i_t,j_t)_{i_m}=x(i_1,j_1,\ldots,i_t,j_t)_{i_m},
		\]
		\[
		y(i_1,j_1,\ldots,i_t,j_t)_{j_{m-1}}=\rk[i_{m},j_{m-1}]+ x(i_1,j_1,\ldots,i_t,j_t)_{i_m}
		\]
		for $1 \leq m \leq t+1$. In particular, $	y(i_1,j_1,\ldots,i_t,j_t)_{{j_0}}=r_{\max}$.
	\end{lemma}
Figure~\ref{fig:FinkMiyazaki} illustrates Definition~\ref{def:Nsequences} and Lemma~\ref{lem:miya}. 
	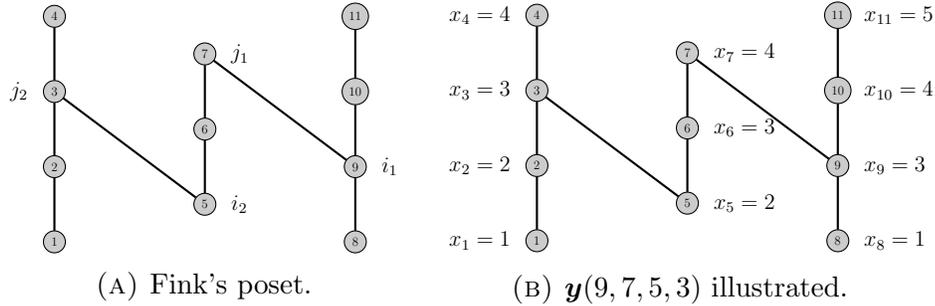
\begin{figure}[h]
	\centering
		\begin{subfigure}{.5\textwidth}
			\centering
		\begin{tikzpicture}[darkstyle/.style={circle,draw,fill=gray!40,minimum size=20}, scale = .5]
		\node[circle,draw=black,fill=white!80!black,minimum size=20, scale=.4] (1) at (0,6) { 4};
		\node[circle,draw=black,fill=white!80!black,minimum size=20, scale=.4] (2) at (0,4) { 3};
		\node[circle,draw=black,fill=white!80!black,minimum size=20, scale=.4] (3) at (0,2) { 2};
		\node[circle,draw=black,fill=white!80!black,minimum size=20, scale=.4] (4) at (0,0) { 1};
		
		\draw[thick] (1)--(2)--(3) -- (4);
		\node[circle,draw=black,fill=white!80!black,minimum size=20, scale=.4] (5) at (4,5) { 7};
		\node[circle,draw=black,fill=white!80!black,minimum size=20, scale=.4] (6) at (4,3) { 6};
		\node[circle,draw=black,fill=white!80!black,minimum size=20, scale=.4] (7) at (4,1) { 5};

		\draw[thick] (5) -- (6) -- (7);
		
		\node[circle,draw=black,fill=white!80!black,minimum size=20, scale=.4] (8) at (8,6) {11 };
		\node[circle,draw=black,fill=white!80!black,minimum size=20, scale=.4] (9) at (8,4) { 10};
		\node[circle,draw=black,fill=white!80!black,minimum size=20, scale=.4] (10) at (8,2) {9 };
		\node[circle,draw=black,fill=white!80!black,minimum size=20, scale=.4] (11) at (8,0) { 8};
		
		\draw[thick] (8)--(9)--(10) -- (11);
		\draw[thick] (2) -- (7) ;
		\draw[thick] (5) -- (10) ;
		
		\node[right, scale=.7] at (8.5,2) {$i_1$};
		\node[right, scale=.7] at (4.5,5) {$j_1$};
		\node[right, scale=.7] at (4.5,1) {$i_2$};
		\node[left, scale=.7] at (-.5,4) {$j_2$};
		\end{tikzpicture}
		\caption{Fink's poset.}
		\label{fig:Fink}	
		\end{subfigure}%
		\begin{subfigure}{.5\textwidth}
			\centering
					\begin{tikzpicture}[darkstyle/.style={circle,draw,fill=gray!40,minimum size=20}, scale = .5]
		\node[circle,draw=black,fill=white!80!black,minimum size=20, scale=.4] (1) at (0,6) { 4};
		\node[circle,draw=black,fill=white!80!black,minimum size=20, scale=.4] (2) at (0,4) { 3};
		\node[circle,draw=black,fill=white!80!black,minimum size=20, scale=.4] (3) at (0,2) { 2};
		\node[circle,draw=black,fill=white!80!black,minimum size=20, scale=.4] (4) at (0,0) { 1};
		
		\draw[thick] (1)--(2)--(3) -- (4);
		\node[circle,draw=black,fill=white!80!black,minimum size=20, scale=.4] (5) at (4,5) { 7};
		\node[circle,draw=black,fill=white!80!black,minimum size=20, scale=.4] (6) at (4,3) { 6};
		\node[circle,draw=black,fill=white!80!black,minimum size=20, scale=.4] (7) at (4,1) { 5};

		\draw[thick] (5) -- (6) -- (7);
		
		\node[circle,draw=black,fill=white!80!black,minimum size=20, scale=.4] (8) at (8,6) {11 };
		\node[circle,draw=black,fill=white!80!black,minimum size=20, scale=.4] (9) at (8,4) { 10};
		\node[circle,draw=black,fill=white!80!black,minimum size=20, scale=.4] (10) at (8,2) {9 };
		\node[circle,draw=black,fill=white!80!black,minimum size=20, scale=.4] (11) at (8,0) { 8};
		
		\draw[thick] (8)--(9)--(10) -- (11);
		\draw[thick] (2) -- (7) ;
		\draw[thick] (5) -- (10) ;
		\node[left, scale = .7] at (-0.5,0) {$x_1 = 1$};
		\node[left, scale = .7] at (-0.5,2) {$x_2 = 2$};
		\node[left, scale = .7] at (-0.5,4) {$x_3 = 3$};
		\node[left, scale = .7] at (-0.5,6) {$x_4 = 4$};
		
		\node[right, scale = .7] at (4.5,1) {$x_5 = 2$};
		\node[right, scale = .7] at (4.5,3) {$x_6 = 3$};
		\node[right, scale = .7] at (4.5,5) {$x_7 = 4$};
		
		\node[right, scale = .7] at (8.5,0) {$x_8 = 1$};
		\node[right, scale = .7] at (8.5,2) {$x_9 = 3$};
		\node[right, scale = .7] at (8.5,4) {$x_{10} = 4$};
		\node[right, scale = .7] at (8.5,6) {$x_{11} = 5$};
		\end{tikzpicture}
			\caption{$\y(9,7,5,3)$ illustrated.}			
			\label{fig:BadPoint}
		\end{subfigure}
		\caption{Fink's poset and the minimal element $\y$ illustrated.}
		\label{fig:FinkMiyazaki}
	\end{figure}

	Now, we can introduce the characterization of Miyazaki:
	\begin{lemma}[{\cite[Thm 3.9]{Miyazaki}}]
		\label{lem:Miyazaki-zigzag}
		Let $\poset$ be a poset and $r= \operatorname{codeg}(\order(\poset))$.
		Then $\poset$ is level is and only if $r_{\max}=r$.
	\end{lemma}
	
	In this paper, we characterize levelness of order polytopes in terms of weighted digraphs.
	Given a poset $\poset$, we define the \emph{Hasse graph $H(\poset)$} of $\poset$ to be the digraph with nodes coming from $\posetTopBottom$ and with directed, weighted edges $(i,j,-1)$ and $(j,i,1)$, where $ i \lessdot j$.
	In our language, a sequence $i_1,j_1,\ldots,i_t,j_t$ with condition $N$ can be reinterpreted as a path $\path$ in $H(\poset)$ from $\infty$ to $-\infty$ with the up-edges (or the down-edges) coming from longest chains in $[i_m,j_m]$ (or $[i_{m+1},j_{m}]$), where $j_0=\infty$ and $i_{t+1}=-\infty$, and we say that such a path satisfies condition $N$.
	Moreover, we set $r(\path):=r(i_1,j_1,\ldots,i_t,j_t)$ and $\y(\path):=\y(i_1,j_1,\ldots,i_t,j_t)$. We chose the weights of the edges, so that $-r(\path)$ equals the weighted length.
	
	\begin{remark}
		This special minimal element $\y(\path)$ has the property that 
		for every up (or down) interval $[i,j]$ in the path we have $y(\path)_i - y(\path)_j = \rk [i,j]$. 
		This is a direct consequence of Lemma \ref{lem:miya} and a brief computation.
	\end{remark}
	
	Now, we can characterize the level property as the following.
	\begin{proposition}
		\label{prop:Miyazaki-zigzag}
		Let $\poset$ be a finite poset and $r= \operatorname{codeg}(\order(\poset))$.
		Then $\poset$ is not level if and only if there exists a path $\path$ in $H(\poset)$ with condition $N$ such that $r(\path) > r$.
	\end{proposition}
	
	\begin{proof}
		This directly follows from Lemmas \ref{lem:Miyazaki-zigzag} and \ref{lem:miya}.
	\end{proof}
	
	\begin{remark}
		Even if $\poset$ is level with $r= \operatorname{codeg}(\order(\poset))$, there may exist a path in $H(\poset)$ of length $>r$.
		The interested reader might construct their favorite counter-example.
	\end{remark}

	\section{A New Characterization of Levelness}
	\label{sec:NewAlgo} 
	In this section, we introduce an algorithm for checking levelness of order polytopes. 
	First, we need to associate a weighted digraph to a poset $\poset$ together with a subposet $\poset' \subset \posetTopBottom$, where we require that $i \lessdot_{\poset'} j$ implies $i \lessdot_{\poset} j$.
	\begin{definition}
	\label{def:WeightedGraphs}
		Let $\poset$ be a finite poset and let $\poset'$ be a subposet of $\posetTopBottom$ such that $i \lessdot_{\poset'} j$ implies $i \lessdot_{\poset} j$. Let $\graph{\poset, \poset'} = (\posetTopBottom, E)$ be the weighted digraph with weighted, directed edges:
		\begin{enumerate}
			\item $(i,j,-1) \in E$ if and only if $j \gtrdot_{\posetTopBottom} i$;
			\item $(j,i,1) \in E$ if and only if $j \gtrdot_{\poset'} i$.
		\end{enumerate}
		Clearly, $\graph{\poset,\posetTopBottom}=H(\poset)$.
		If $\poset$ is clear from the context, we will write $\graph{\poset'}$.
	\end{definition}
	A \emph{negative cycle} is a directed cycle whose sum of weights is negative.  A \emph{wedge of cycles} is a closed directed path (where repetition is allowed) whose sum of weights is negative. 
	An integer point $\x\in\int(\cone(\order(\poset))) \cap \Z^{d+1}$  is \emph{sharp} along a covering pair $j \gtrdot i$ if $x_j = x_i+1$.

	Next, we associate a weighted digraph to every integer point in $\int(\cone(\order(\poset)))$. The following lemma shows that the associated digraph does not have any negative cycles. 
	\begin{lemma}
		\label{lem:NoNegCycles}
		Let $\b \in \int(\cone(\order(\poset))) \cap \Z^{d+1}$ be given. Then the weighted digraph $\Gamma_{\b}$ whose nodes are given by the elements of $\posetTopBottom$ and whose weighted, directed edges are
		\begin{itemize}
			\item $\left \{ (i,j,-1)\colon i \lessdot j\right\}  \cup \left \{(j,i, 1) \colon i \lessdot j \text{, } b_j - b_i=1\right\}$ ,
			\item $\{(-\infty, i, -1): i \gtrdot -\infty \} \cup \{( i,-\infty, 1): i \gtrdot -\infty, b_i  =1\}$,
			\item and $\{(i,\infty, -1): \infty  \gtrdot i\} \cup \{(\infty, i, 1): \infty \gtrdot i, b_i  =\max_j {b_j}\}$.
		\end{itemize}
		does not have any negative cycles. In particular, every subgraph contains no negative cycles. 
	\end{lemma}
	
	\begin{proof}
		For $i, j \in \posetTopBottom$, let $u(i,j)$ denote a directed path from $i$ to $j$ where $i<j$ and let $d(l,k)$ denote a directed path from $l$ to $k$ where $l>k$.
		Let $u(i_1,i_2)$, $d(i_2,i_3)$, $u(i_3,i_4)$, $\dots$, $d(i_s,i_1)$ be a directed cycle in $\Gamma_{\b}$ with $i_1< i_2$, $i_3<i_2$, $\dots$, $i_1<i_s$. \rem{In particular, $p(i_1,i_2)$ is an up path, while $p(i_3,i_2)$ is a down path, etc.} We first remark that the weights of the down-paths $d(i,i+1)$ are given by  
		\[
		\stackrel{=1}{\overbrace{b_{i} - b^{(1)}}}+ \stackrel{=1}{\overbrace{b^{(1)} - b^{(2)}}} + \dots - b^{(r)}+ \stackrel{=1}{\overbrace{b^{(r)} - b_{i+1}}} = b_{i} - b_{i+1},
		\]
		where we set $b_{\infty} = \max_k b_k +1$ and $b_{-\infty} = 0$.
		Therefore, the sum of the weights in the cycle is equal to
		\begin{align*}
			&(b_{i_2} - b_{i_3}) +  \cdots + (b_{i_s} - b_{i_1}) - \length{i_1,i_2}-  \cdots -\length{i_{s-1},i_s} \\
			= &(b_{i_2} - b_{i_1}) +  \cdots + (b_{i_s} - b_{{i_{s-1}}}) - \length{i_1,i_2}-  \cdots -\length{i_{s-1},i_s} \geq 0,
		\end{align*}
		since $\b \in \int(\coneOPi)$ implies that $b_j - b_i \geq \length{i,j}$ for all $i<j$, where $\length{i,j}$ is the length of the path from $i$ to $j$. 
	\end{proof}
	The following theorem uses the \emph{Bellman--Ford algorithm}, which was introduced by Bellman and Ford, see for instance \cite{Bellman}. We are using this algorithm as a black box. Instead of explicitly describing it, we will merely state some basic facts about it:
	\begin{itemize}
		\item The Bellman--Ford algorithm finds the shortest path from a sink to any other node in a weighted digraph. In contrast to other algorithms, it can also deal with negative weights assuming that the digraph does not contain any negative cycles (that can be reached from the starting node), see \cite[Thm. 8.5]{SchrijverCombOptA}.
		\item If there is such a negative cycle, the Bellman--Ford algorithm can detect the negative cycle, \cite[Thm. 8.6]{SchrijverCombOptA}.
		\item The Bellman--Ford algorithm runs in $O(\# V \cdot \# E)$, where $V$ is the vertex set and $E$ is the edge set of the underlying graph, see \cite[Thm. 8.5]{SchrijverCombOptA}.
	\end{itemize}
	Given a path $\path$  in $H(\poset)$ with condition $N$, let $\poset'(\path)$ be the subposet of $\posetTopBottom$ whose covering pairs are given by the up paths of $\path$.
	Now, we give a new characterization of level order polytopes.
	\begin{theorem}
		\label{thm:levelchar}
		Let $\poset$ be a finite poset on $d$ elements and let $r= \operatorname{codeg}(\order(\poset))$. Then $\poset$ is level if and only if for any path $\path$ in $H(\poset)$ with condition $N$ such that $\graph{\poset'(\path)}$ has no negative cycles,
		$\Gamma(\poset'(\path) \cup \{\textnormal{longest chains in }\posetTopBottom \})$ has no negative cycles. 
	\end{theorem}
	\begin{proof}
		%
		To show the first direction, let's assume $\order(\poset)$ is level. 
		Associated to $\Gamma(\poset'(\path))$, there is a $\b \in \int(\coneOPi)\cap \Z^{d+1}$. To see this, run the Bellman--Ford algorithm on $\Gamma(\poset'(\path))$. The Bellman--Ford algorithm minimizes the distance from $-\infty$ to any point $i$. After multiplying all entries by $-1$, the algorithm will return a point $\b$ such that for all covering pairs $ j \gtrdot i$ in $\posetTopBottom$, we have $b_j \geq b_i +1$ and moreover for any weighted edge $(j,i,1)$ in $\Gamma(\poset'(\path))$ we have $b_i \geq b_j - 1$. The first condition implies that $\b \in \int(\coneOPi)\cap \Z^{d+1}$.
		Moreover, it follows from these conditions that for any weighted edge $(j,i,1)$ in $\Gamma(\poset'(\path))$, we have $b_j -b_i=1$. Since $\poset$ is level, there exists a point $\tilde {\b} \in \int(\coneOPi) \cap \Z^{d+1}$ on height $r$ such that $\b - \tilde {\b} \in \coneOPi \cap \Z^{d+1}$. This implies that for every covering pair $j \gtrdot i$ in $\posetTopBottom$, we have 
		\[
		b_j - b_i \geq \tilde{b}_j - \tilde{b}_i.
		\]
		Hence for any weighted edge $(j,i,1)$ in $\Gamma(\poset'(\path))$, we have $\tilde{b_j} -\tilde{b_i}=1$.
		Since $\tilde{\b}$ is on height $r$, we also know that $\tilde{\b}$ is sharp along the longest chains in $\posetTopBottom$.
		Since $\mathcal{G}: = \Gamma(\poset' \cup \{\text{longest chains in } \posetTopBottom \})$ is a subgraph of $\Gamma_{\tilde{\b}}$,
		using Lemma~\ref{lem:NoNegCycles}, we get that 
		$\mathcal{G}$ does not contain a negative cycle. 
		
		We prove the other direction by contraposition. Let's assume that $\poset$ is not level. Then there exists a
		path $\path$ with condition $N$ such that $r(\path)=r_{\max} > r$.
		Moreover,  $\graph{\poset'(\path)}$ is a subgraph of $\Gamma_{\y(\path)}$.
		Hence by using Lemma \ref{lem:NoNegCycles}, it follows that  $\graph{\poset'(\path)}$ has no negative cycle.
		On the other hand, since $\rk(\posetTopBottom)=r$, it follows that
		$\Gamma(\poset'(\path) \cup \{\text{longest chains in }\posetTopBottom \})$ has a negative cycle.
		%
	\end{proof}
	This directly implies the following:
	\begin{corollary}
		\label{cor:GeneralGamma'}
		Let $\poset$ be a finite poset. $\poset$ is level if and only if for all $\Gamma(\poset')$ that do not have a negative cycle, the digraph $\Gamma(\poset' \cup \{\textnormal{longest chains in } \posetTopBottom\})$ does not have a negative cycle.
	\end{corollary}
	\begin{proof}
		One direction directly follows from Theorem~\ref{thm:levelchar}, so we only need to show that if $\poset$ is level and if $\Gamma(\poset')$ does not contain a negative cycle, then  $\Gamma(\poset' \cup \{\textnormal{longest chains in } \posetTopBottom\})$ does not have a negative cycle. However, this follows from the proof of Theorem~\ref{thm:levelchar}.
	\end{proof}
	\begin{remark}
		\label{rem:CorExplanation}
		For practical purposes this corollary is more convenient than the previous characterization. This is due to the fact that it is hard to determine (all) paths with condition $N$. We will use this corollary to give an infinite family of level posets, see Theorem~\ref{thm:OrdSumLevel}.
	\end{remark}
Moreover, we get that --- given the input $\poset$ --- determining the levelness of order polytopes is in $\operatorname{co-NP}$, i.e., the complexity class having a short certificate for rejection. For more about complexity classes, we refer to \cite[Sec. 2.5]{schrijver1986}.
	\begin{corollary}
		\label{cor:co-NP}
		Levelness of order polytopes is in $\operatorname{co-NP}$.
	\end{corollary}
	\begin{proof}
		If $\order(\poset)$ is not level, then there exists a short certificate $\poset' (\path)$ such that $\graph{\poset'(\path)}$ does not have a negative cycle but $\graph{ \poset '(\path) \cup \{\textnormal{longest chains in } \posetTopBottom \}}$ has a negative cycle.
		This will be tested by the Bellman--Ford algorithm in polynomial time, since we need to run the Bellman--Ford algorithm twice, once for $\graph{\poset'(\path)}$ and once for $\graph{ \poset '(\path) \cup \{\textnormal{longest chains in } \posetTopBottom \}}$. Therefore, we can verify non-levelness in polynomial time.
	\end{proof}

	We now explicitly describe the algorithm underlying Corollarly~\ref{cor:GeneralGamma'}:
	\begin{algorithm}
		\label{alg:Miyazaki}	
		\text{ }\newline
		For $\graph{\poset'}\subset H(\poset)$: 
		
		\qquad Run Bellman--Ford for $\graph{\poset'}$
		
		\qquad If negative cycle:
		
		\qquad \qquad 1=1
		
		\qquad Else:
		
		\qquad \qquad Run Bellman--Ford for  $\graph{\poset' \cup \{\textnormal{longest chains in } \posetTopBottom\}}$
		
		\qquad \qquad If negative cycle:
		
		
		
		
		
		
		\qquad \qquad\qquad  Return NOT LEVEL
		
		\qquad \qquad Else:
		
		\qquad \qquad\qquad 1=1 
		
		\qquad Return LEVEL
	\end{algorithm}
	
	\begin{theorem}
		\label{thm:LevelCharAlgMiyazaki}
		A poset $\poset$ is level if and only if Algorithm~\ref{alg:Miyazaki} returns level.
	\end{theorem}
	\begin{proof}
		This directly follows from Corollary~\ref{cor:GeneralGamma'}.
	\end{proof}
	
	\section{A Necessary Condition of Ene, Herzog, Hibi, and Saeedi Madani}
	\label{sec:NecessaryCondition}
	
	We now want to show that \cite[Thm 4.1]{EneHerzogHibi} is a special case of Corollary~\ref{cor:GeneralGamma'}. We first need to define the depth and the height of an element, where we follow again \cite{EneHerzogHibi}. The \emph{height of an element $i\in \poset$}, denoted $\height(i)$ is the maximum length of a chain in $\posetTopBottom$ descending from $i$. Similarly, we define the \emph{depth of an element $i$}, denoted $\depth(i)$, to be the maximum length of a chain in $\posetTopBottom$ ascending from $i$.

	They show that the following is a necessary condition for levelness:
	\begin{theorem}[{\cite[Thm. 4.1]{EneHerzogHibi}.}]
		Suppose $\poset$ is level. Then
		\begin{equation}
			\label{eq:heightdepth}
			\operatorname{height}(j) + \operatorname{depth} (i) \leq \operatorname{rank} (\posetTopBottom) +1
		\end{equation}
		for all $j \gtrdot i \in \poset$.
	\end{theorem}
	Our next result shows that this is weaker than Corollary~\ref{cor:GeneralGamma'}. In fact, it is equivalent to Corollary~\ref{cor:GeneralGamma'} where $\poset'$ is a single edge.
	
	\begin{theorem}
		\label{thm:HibiConditionIsEquivalent}
		Let $\poset$ be a finite poset and $r= \operatorname{codeg}(\order(\poset))$.
		The following are equivalent:
		\begin{enumerate}
			\item inequality (\ref{eq:heightdepth}) is satisfied by all covering pairs
			\item for all Hasse edges $j\gtrdot i \in \poset$ there is an integer point $\x \in r\int(\order(\poset))$ such that  $x_j = x_i +1$.
		\end{enumerate}
	\end{theorem}
	
	\begin{proof}
		Let's assume that all covering pairs satisfy (\ref{eq:heightdepth}) and fix a covering pair $j \gtrdot i$. We remark that $\operatorname{rank} (\posetTopBottom)$ equals the codegree $r$ of the order polytope $\order(\poset)$. Thus, we need to show that there exists an integer point $\x \in r \int(\order(\poset))$ such that  $x_j = x_i +1$. To create such an $\x$, we can label the elements in $\poset$ using labels from $\{1, 2, \dots, r-1\}$.   We first label $x_j = \operatorname{height}(j)$ and hence $x_i = \operatorname{height}(j) -1$. For $k \in \poset \setminus \{i,j \}$ we label $x_k = - \infty$, and then we recursively relabel by
		\begin{equation}
			\label{eq:label}
			x_k = \begin{cases}  \max\{\height(k), x_i + \length{[i,k]}\} &\text{ if }k>i,\\
				\height(k) &\text{ otherwise.}\\
			\end{cases}
		\end{equation}
		To show that this indeed gives an interior integer point in $r \int \order(\poset)$, we need to show that $r> x_k \geq \height (k)$ for all $k$. We say a label $x_k$ is well-defined if  it satisfies this condition. There are two cases:
		\begin{enumerate}
			\item $k>i$, then (\ref{eq:heightdepth}) ensures that (\ref{eq:label}) only yields well-defined labels;
			\item $k \ngtr i$, then the recursive definition gives us $\height(k)$, which by definition is well-defined.
		\end{enumerate}
		This proves the first direction. 
		
		Now let assume that for all Hasse edges ($j\gtrdot i$) in $\poset$ there exists an integer point $\x \in r \int(\order(\poset))$ such that  $x_j = x_i +1$. Let's fix a covering pair $j \gtrdot i$. Then we have an integer point $\x \in r\int(\order(\poset))$ with $x_j = x_i +1$ and it follows that $\height(j)\leq x_j$. Since we have an integer point in the interior of $r \order(\poset)$, we also get that
		\[
		\depth(i) \leq \rk(\posetTopBottom)- x_i . 
		\]
		Putting everything together, we obtain
		\[
		\height(j) + \depth (i)\leq   \rk(\posetTopBottom) +1 ,
		\]
		as desired.
	\end{proof}

	However, this result is not sufficient. The following example was related to us by Alex Fink, see Figure~\ref{fig:Fink}.

	\begin{remark}
		\label{rem:Fink}
		Let $\poset$ be the poset from Figure~\ref{fig:Fink}. We have that $\codeg(\order(\poset)) = 5$. Moreover, for any covering pair $i \lessdot j$ in $\poset$, there is a minimal element $\x$ on height $5$ with $x_i +1 = x_j$. Thus, by Theorem~\ref{thm:HibiConditionIsEquivalent} the condition of \cite[Thm. 4.1]{EneHerzogHibi} is satisfied. However, $\poset$ is \emph{not} level. The minimal element $\y(9,7,5,3)$ is on height $6$, see Figure~\ref{fig:BadPoint}.
	\end{remark}
	
	\section{Series-Parallel Posets}
	\label{sec:SeriesParallelPosets}

	The goal of this section is to describe a new family of level posets. The main character of this section is the \emph{ordinal sum}.  We follow the notation of~\cite[Sec. 3.2] {StanleyEC1}. Let $\poset_1$ and $\poset_2$ be two posets. Then their ordinal sum $\poset_1 \triangleleft \poset_2$ is the poset with elements from the union $\poset_1 \cup \poset_2$ and with relations $s\leq t$ if
	\begin{itemize}
		\item $s, t \in \poset_2$ with $s\leq_{\poset_2} t$, or
		\item $s, t \in {\poset_1}$ with $s\leq_{\poset_1} t$, or
		\item $s \in {\poset_1}$ and $t \in \poset_2$.
	\end{itemize}
	Posets that can be built up up as ordinal sums of posets are called \emph{series-parallel posets}.
	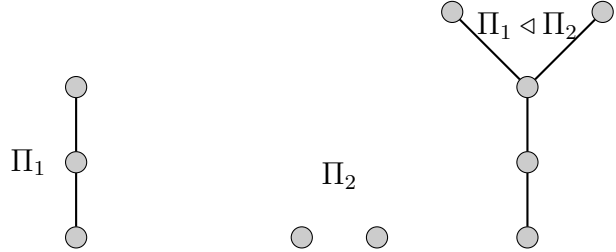
\begin{figure}[h]
		\center
		\begin{tikzpicture}[darkstyle/.style={circle,draw,fill=gray!40,minimum size=20}, scale = .5]
		\node[circle,draw=black,fill=white!80!black,minimum size=20, scale=.4] (1) at (0,4) { };
		\node[circle,draw=black,fill=white!80!black,minimum size=20, scale=.4] (2) at (0,2) { };
		\node[circle,draw=black,fill=white!80!black,minimum size=20, scale=.4] (3) at (0,0) { };
		\draw[thick] (1)--(2)--(3);
		\node[left] at (-.5,2) {$\poset_1$};
		\node[circle,draw=black,fill=white!80!black,minimum size=20, scale=.4] (5) at (8,0) { };
		\node[circle,draw=black,fill=white!80!black,minimum size=20, scale=.4] (6) at (6,0) { };
		\node[above] at (7,1) {$\poset_2$};
		\node[circle,draw=black,fill=white!80!black,minimum size=20, scale=.4] (7) at (12,4) { };
		\node[circle,draw=black,fill=white!80!black,minimum size=20, scale=.4] (8) at (12,2) { };
		\node[circle,draw=black,fill=white!80!black,minimum size=20, scale=.4] (9) at (12,0) { };
		\draw[thick] (7)--(8)--(9);
		\node[above] at (12,5) {$\poset_1 \triangleleft \poset_2$};
		\node[circle,draw=black,fill=white!80!black,minimum size=20, scale=.4] (10) at (10,6) { };
		\node[circle,draw=black,fill=white!80!black,minimum size=20, scale=.4] (11) at (14,6) { };
		\draw[thick] (10)--(7);
		\draw[thick] (11)--(7);
		\end{tikzpicture}
		\caption{Ordinal sum of a chain of length 3 and an antichain of length 2.}
	\end{figure}
	
	We want to show the following result:
	\begin{theorem}
		\label{thm:OrdSumLevel}
		The ordinal sum $\poset ={\poset_1} \triangleleft \poset_2$ of two posets $\poset_1$, $\poset_2$ is level if and only if both $\poset_1$ and $\poset_2$ are level.
	\end{theorem}
	\begin{proof}
		
		We prove the first direction by contraposition. So let's assume that $\poset_1$ is not level. 
		By Corollary \ref{cor:GeneralGamma'},
		there exists a weighted digraph 
		$\Gamma_{\poset_1}$ with nodes coming from $\addBottom{\overline{\poset_1}}$ which does not contain a negative cycle, but the weighted directed graph $\Gamma_{\poset_1} \cup \{ \text{longest chains in } \addBottom{\overline{\poset_1}}\}$ has a negative cycle.
		However, we also get that $\Gamma_{\poset}$ has a weighted digraph with up-edges of weight $1$ only coming from up-edges of $\Gamma_{\poset_1}$ does not contain a negative cycle, but $\Gamma_{\poset} \cup \{ \text{longest chains in } \posetTopBottom\}$ contains one, proving that $\poset_1 \triangleleft \poset_2$ is not level. The case where $\poset_2$ is not level follows analogously.

		We prove the other direction again by contraposition. So let's assume that $\poset_1 \triangleleft \poset_2$ is not level. 	By Corollary \ref{cor:GeneralGamma'}, there exists a weighted digraph $\Gamma$ with nodes coming from $\posetTopBottom$ such that $\Gamma$ does not have a negative cycle and  $\Gamma \cup \{\textnormal{longest chains in } \posetTopBottom\}$ has a negative cycle,
		where $\Gamma \cup \{\textnormal{longest chains in } \posetTopBottom\}$ is the weighted digraph obtained from $\Gamma$ by adding down edges of weight 1 along longest chains in $\posetTopBottom$.
		In order to show that either $\poset_1$ or $\poset_2$ are not level, we will construct graphs $\Gamma_{\poset_1}$ and $\Gamma_{\poset_2}$ without negative cycles such that adding  down edges of weight 1 along longest chains creates a negative cycle. The following two quotient maps will be essential for this:
		\begin{align*}
					&\addBottom{\overline{{\poset_1} \triangleleft {\poset_2}}} \stackrel{q_1}{ \twoheadrightarrow} \addBottom{\overline{{\poset_1} \triangleleft {\poset_2}}} / (p_2 \sim p_2' \sim \infty) \quad\cong \addBottom{\overline{{\poset_1}}} \\
			&\addBottom{\overline{{\poset_1} \triangleleft {\poset_2}}}  \stackrel{q_2}{\twoheadrightarrow} \addBottom{\overline{{\poset_1} \triangleleft {\poset_2}}} / (p_1 \sim p_1' \sim -\infty) \,\cong \addBottom{\overline{{\poset_2}}},
		\end{align*}
		where $p_1,p_1' \in \poset_1$ and $p_2,p_2' \in \poset_2$.

		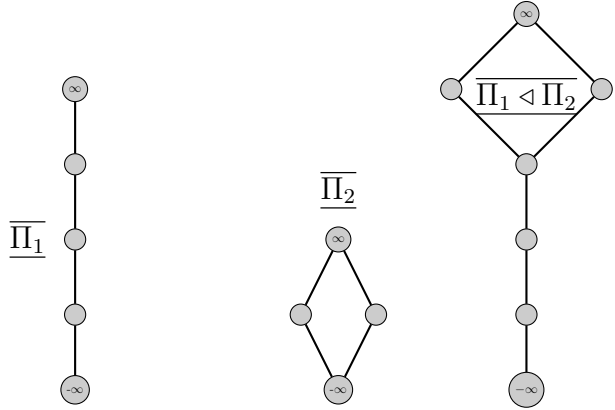
\begin{figure}[h]
			\center
			\begin{tikzpicture}[darkstyle/.style={circle,draw,fill=gray!40,minimum size=20}, scale = .5]
			\node[circle,draw=black,fill=white!80!black,minimum size=20, scale=.4] (0) at (0,6) {$\infty$ };
			\node[circle,draw=black,fill=white!80!black,minimum size=20, scale=.4] (1) at (0,4) { };
			\node[circle,draw=black,fill=white!80!black,minimum size=20, scale=.4] (2) at (0,2) { };
			\node[circle,draw=black,fill=white!80!black,minimum size=20, scale=.4] (3) at (0,0) { };
			\node[circle,draw=black,fill=white!80!black,minimum size=20, scale=.4] (4) at (0,-2) { -$\infty$ };
			\draw[thick] (0)--(1)--(2)--(3)--(4);
			\node[left] at (-.5,2) {$\addBottom{\overline{{\poset_1}}}$};
			\node[circle,draw=black,fill=white!80!black,minimum size=20, scale=.4] (12) at (7,2) {$\infty$};
			\node[circle,draw=black,fill=white!80!black,minimum size=20, scale=.4] (5) at (8,0) { };
			\node[circle,draw=black,fill=white!80!black,minimum size=20, scale=.4] (6) at (6,0) { };
			\node[circle,draw=black,fill=white!80!black,minimum size=20, scale=.4] (13) at (7,-2) {-$\infty$};
			\node[above] at (7,2.5) {$\addBottom{\overline{{\poset_2}}}$};
			\draw[thick] (5)--(12) --(6);
			\draw[thick] (5)--(13) --(6);
			\node[circle,draw=black,fill=white!80!black,minimum size=20, scale=.4] (7) at (12,4) { };
			\node[circle,draw=black,fill=white!80!black,minimum size=20, scale=.4] (8) at (12,2) { };
			\node[circle,draw=black,fill=white!80!black,minimum size=20, scale=.4] (9) at (12,0) { };
			\draw[thick] (7)--(8)--(9);
			\node[above] at (12,5) {$\addBottom{\overline{{\poset_1} \triangleleft {\poset_2}}}$};
			\node[circle,draw=black,fill=white!80!black,minimum size=20, scale=.4] (10) at (10,6) { };
			\node[circle,draw=black,fill=white!80!black,minimum size=20, scale=.4] (11) at (14,6) { };
			\node[circle,draw=black,fill=white!80!black,minimum size=20, scale=.4] (14) at (12,8) {$\infty$ };
			\node[circle,draw=black,fill=white!80!black,minimum size=20, scale=.4] (15) at (12,-2) {$-\infty$ };
			\draw[thick] (10)--(14)--(11);
			\draw[thick] (9)--(15);
			\draw[thick] (10)--(7);
			\draw[thick] (11)--(7);
			\end{tikzpicture}
			\caption{Original poset (on the right) and the two quotient posets (on the left and in the middle).}
			\label{fig:QuotientPoset}
		\end{figure}

		Note that these quotient maps also induce weighted directed graphs $\Gamma_{\poset_1}$ and $\Gamma_{\poset_2}$ on the underlying posets $\addBottom{\overline{{\poset_1}}}$ and $\addBottom{\overline{{\poset_2}}}$, respectively. We will show the following:
		\begin{enumerate}
			\item Both $\Gamma_{\poset_1}$ and $\Gamma_{\poset_2}$ do not have a negative cycle
			\item Either $\Gamma_{\poset_1} \cup \{\text{longest chains in }{\addBottom{\overline{{\poset_1}}}} \}$ or $\Gamma_{\poset_2} \cup \{\text{longest chains in } {\addBottom{\overline{{\poset_2}}}} \}$ or both have a negative cycle.
		\end{enumerate}
		This implies that either $\poset_1$ or $\poset_2$ or both cannot be level proving the claim. The first claim follows by contraposition. If either $\Gamma_{\poset_1}$ or $\Gamma_{\poset_2}$ had a negative cycle, then one can \emph{lift} this cycle to obtain a negative cycle in $\Gamma$. This is due to the fact that every maximal element in $\poset_1$ is comparable to every minimal element in $\poset_2$, together with the fact that every up-edge has the same weight, namely $-1$. 
		
		Now let's prove the second claim. We remark that longest chains in $\addBottom{\overline{{\poset_1} \triangleleft {\poset_2}}}$ are concatenations of longest chains in $\addBottom{\overline{{\poset_1}}}$ and longest chains in $\addBottom{\overline{{\poset_2}}}$ and vice versa. This means that 
		\[
		\operatorname{im_{\poset_1}}(\Gamma \cup \{ \text{longest chains in }{\addBottom{\overline{{\poset_1} \triangleleft {\poset_2}}}}\}) = \operatorname{im}_{\poset_1}(\Gamma)\cup \{\text{longest chains in } {\addBottom{\overline{{\poset_1}}}}\}
		\]
		and 
		\[
		\operatorname{im_{\poset_2}}(\Gamma \cup \{ \text{longest chains in }{\addBottom{\overline{{\poset_1} \triangleleft {\poset_2}}}}\}) = \operatorname{im}_{\poset_2}(\Gamma)\cup \{\text{longest chains in } {\addBottom{\overline{{\poset_2}}}}\},
		\]
		where $\operatorname{im_{\poset_1}}$ (or $\operatorname{im_{\poset_2}})$ denotes the image of the quotient map onto $\addBottom{\overline{{\poset_1}}}$ (or $\addBottom{\overline{{\poset_2}}}$). Moreover, if a negative cycle of $\Gamma \cup \{\text{longest chains in }{\addBottom{\overline{{\poset_1} \triangleleft {\poset_2}}}} \}$ is entirely contained in $q_1^{-1}(\addBottom{{{\poset_1}}})~\cup~\{\text{min. elt's of } \poset_2\}$ or $q_2^{-1}({\overline{{\poset_2}}})~\cup~\{\text {max. elt's of } \poset_1\}$, then clearly the image also has a negative cycle. (Caveat: After forming the quotient map, the cycle might become a wedge of cycles. But since the total weight of the original cycle is the sum of the weights of the cycle in the image, at least one of these cycles in the wedge has to be negative.)
		
		So we only need to consider the case where a negative cycle contains edges contained in $\poset_1$ and in $\poset_2$. We can cover the cycle into the part, whose edges are entirely in $q_1^{-1}(\addBottom{{{\poset_1}}})~\cup~\{\text{min. elt's of } \poset_2\}$, and a part whose edges are in $q_2^{-1}({\overline{{\poset_2}}})~\cup~\{\text {max. elt's of } \poset_1\}$. Note that the edges between $\poset_1$ and $\poset_2$ appear in both parts. Therefore, the total weight $w$ of the cycle equals
		\[
		0>w = w_{\addBottom{\overline{{\poset_1}}}} +w_{\addBottom{\overline{{\poset_2}}}} - w_{{\poset_1}{\poset_2}},
		\]
		where $w_{\addBottom{\overline{{\poset_1}}}}$ and $w_{\addBottom{\overline{{\poset_2}}}}$ are the weights of the parts in the preimage of $\addBottom{\overline{{\poset_1}}}$ and $\addBottom{\overline{{\poset_2}}}$, respectively. The weight of the connecting edges between $\poset_1$ and $\poset_2$ is denoted $w_{{\poset_1}{\poset_2}}$. This weight is 0, since there are as many up- as there are down-edges and the weights are $-1$ and $1$, respectively. Therefore, either $w_{\overline{{\poset_1}}}$ or $w_{\overline{{\poset_2}}}$ or both are negative. If $w_{\addBottom{\overline{{\poset_1}}}}$ is negative, applying the quotient map gives us a wedge of cycles in $\addBottom{\overline{{\poset_1}}}$ with negative weight. Hence it contains at least one negative cycle. The case where $w_{\addBottom{\overline{{\poset_2}}}}$ is negative is similar. Therefore, we have seen that either $\poset_1$ or $\poset_2$ is not level proving the claim.

	\end{proof}
	
	For the remainder of this section, let $\poset={\poset_1} \triangleleft \poset_2$. We will first give a geometric description of the order polytope and the chain polytope of $\poset$ in terms of the order and chain polytopes of $\poset_1$ and $\poset_2$, respectively.
	
	\begin{lemma}
		\label{lem:OrdSumChainDesc}
		Let $\chain({\poset})$, $\chain({\poset_1})$, $\chain({\poset_2})$ be the chain polytopes of $\poset$, ${\poset_1},$ and $\poset_2$, respectively. Then  
		\begin{equation}
			\label{eq:ordsumchainpoly}
			\chain({\poset})=\operatorname{conv}\{\chain({\poset_1}) \times \mathbf 0_{\poset_2} \cup \mathbf 0_{\poset_1} \times \chain({\poset_2})  \} =  \chain({\poset_1}) \oplus \chain({\poset_2}),
		\end{equation}
		where $\oplus$ is the \emph{free sum} of $\chain({\poset_1})$ and $\chain({\poset_2})$.
	\end{lemma}
	\begin{proof}
		By Theorem \ref{thm:VerticesChainPoly}, the vertices of the chain polytope are given by the indicator vectors of antichains. Now one notices that no antichain can contain elements from both $\poset_1$ and $\poset_2$.
	\end{proof}
	
	\begin{lemma}
		\label{lem:OrdSumOrderDesc}
		Let $\order({\poset}), \order({\poset_1}), \order({\poset_2})$ be the order polytopes of $\poset, {\poset_1}, \poset_2$, respectively. Then  
		\begin{equation}
			\label{eq:ordsumorderpoly}
			\order({\poset})=\operatorname{conv}\{\mathbf 1_{\poset_2} \times \order({\poset_1})\cup\order({\poset_2}) \times \mathbf  0_{\poset_1}  \}.
		\end{equation}
	\end{lemma}
	\begin{proof}
		By Corollary \ref{cor:VerticesOrderPoly}, the vertices of the order polytope are given by the indicator vectors of filters. Now one notices that as soon as a filter contains an element of $\poset_1$, it contains \emph{all} elements of $\poset_2$.
	\end{proof}
	
	Moreover, we have:
	\begin{proposition}
		\label{prop:HStarOrdinalSum}
		Let $\poset, {\poset_1}, \poset_2$, be finite posets where $\poset := {\poset_1} \triangleleft \poset_2$. Moreover, let $h^{\ast}_{\poset}, h^{\ast}_{\poset_1}, h^{\ast}_{\poset_2}$ be the $h^{\ast}$-polynomial of the Ehrhart series of the corresponding order polytopes. Then
		\begin{equation*}
			\label{eq:hstarordsum}
			h^{\ast}_{\poset} =  h^{\ast}_{\poset_1}  h^{\ast}_{\poset_2}.
		\end{equation*} 
	\end{proposition}
	\begin{proof}
		By Theorem~\ref{thm:EhrhartPolyofOrder} the Ehrhart series of the chain polytope of a poset $\poset$ is the same as the Ehrhart series of the order polytope of $\poset$.In~\cite[Lem 3.2]{HibiHigashitani}, Hibi and Higashitani show that if the free sum, $P \oplus Q$, of two lattice polytopes $P$, $Q$ both containing the origin has the integer-decomposition property, then  $h^{\ast}_{P \oplus Q} = h^{\ast}_P h^{\ast}_Q$.
Now using Lemma~\ref{lem:OrdSumChainDesc} together with~\cite[Lem 3.2]{HibiHigashitani} implies the result. Note that every chain polytope and every order polytope has a unimodular triangulation and thus has the integer-decomposition property. For the order polytope, we directly get a regular, unimodular, flag triangulation by taking the standard triangulation of the cube and restricting it to the order polytope. For the chain polytope,  Stanley~\cite{StanleyOrderPoly} constructs such a regular, unimodular, flag triangulation.
	\end{proof}
	
	\begin{remark}
		\label{rem:HibiExampleSameHStar}
		In \cite{HibiStraightening}, Takayuki Hibi gives an example of two order polytopes $\order(\poset_1)$, $\order(\poset_2)$ where both have the same $h^{\ast}$-polynomial, but $\poset_1$ is level and $\poset_2$ is not level. This shows that the level property cannot be characterized by the $h^{\ast}$-polynomial. We remark that Theorem~\ref{thm:OrdSumLevel} together with Proposition \ref{prop:HStarOrdinalSum} gives a way to create infinitely many such examples $\poset_1 \triangleleft \poset_3$ and $\poset_2 \triangleleft \poset_3$, where $\poset_3$ is any level poset.
	\end{remark}

	\section{Connected components of level posets}
	\label{sec:ConnectedComponents}\rem{
		\christiancomment{$P, Q$ IDP + level then $P \times Q$ IDP +
			level. Florian will deliver example that both are necessary}
		
		\christiancomment{given $P$, what can we say about the set of $k$
			s.t. $kP$ level?}
		
		\flocomment{since we use $r$ for codegree in the other sections, I will change $l$ to $r$ here. $r \mapsto s$,$l \mapsto r$ }}
	In this section, we discuss connected components of level posets.
	Any connected component of a Gorenstein poset is Gorenstein.
	This fact naturally leads us consider whether any connected component of a level poset is level.
	However, this is not true in general.
	From the following result we know that there exists a level poset such that a connected component of the poset is not level.
	
	\begin{theorem}[{\cite[Theorem 4.7]{EneHerzogHibi}}]
		\label{thm:chain}
		Let $\poset$ be a poset on $d$ elements and let $C_{s}$ be a totally ordered set with $s$ elements. Then the poset on the set $\poset \cup C_{s}$ ,where elements from $\poset$ and $C_{s}$ are incomparable, is level for all $s \gg 0$.
	\end{theorem}
	
	We give an explicit bound for $s$ appearing in Theorem~\ref{thm:chain}.
	\begin{theorem}
		\label{thm:chainbound}
		Let $\poset$ be a poset on $d$ elements and let $C_{s}$ be a totally ordered set with $s$ elements. Then the poset on the set $\poset \cup C_{s}$, where elements from $\poset$ and $C_{s}$ are incomparable, is level for all $s\geq d$.
	\end{theorem}
	
	In order to prove this theorem, we consider a more general class of lattice polytopes containing any order polytope.
	\begin{definition}
		We say that a polytope  $\polytope \subset \R^d$ is {\em alcoved} if $\polytope$ is  an intersection of some half-spaces bounded by the hyperplanes 
		$$
		H_{i,j}^{m}=\{(z_1,\ldots,z_d) \in \R^d : z_i-z_j=m\} \textnormal{\ for\ } 0 \leq i < j \leq d, m \in \Z,
		$$ 
		where $z_0=0$.  
	\end{definition}
	It is known that any order polytope is alcoved. After a unimodular change of coordinates, every chain polytope is alcoved, too. Furthermore, any alcoved polytope possesses the integer-decomposition property.

	For a lattice polytope $\polytope =\{\x \in \R^d : A\x \leq \b \} \subset \R^d$,
	we set
	$\polytope^{(1)}= \{\x \in \R^d: A\x \leq \b-1\}$. 
	\begin{remark}
		If $\polytope =\{\x \in \R^d : A\x \leq \b \} \subset \R^d$ is an alcoved polytope for some $m \times d$ integer matrix $A$ and some integer vector $\b \in \Z^m$, then since $A$ is a totally unimodular matrix,
		$\Pc^{(1)}$ is a lattice polytope.
		In particular, one has 
		$\Pc^{(1)}=\conv(\int(\polytope) \cap \Z^d)$.
	\end{remark}
	
	\begin{lemma}
		\label{alcoved:lemma}
		Let $\polytope \subset \R^d$ be an alcoved polytope.
		Then for any positive integer $k$, $k\polytope$ and $\polytope^{(1)}$ are alcoved.
	\end{lemma}	
	\begin{proof}
		Since $\polytope$ is alcoved, $\polytope$ is a polytope given by inequalities of the form $b_{ij} \leq z_i-z_j \leq c_{ij}$, for some collection of integer parameters $b_{ij}$ and $c_{ij}$.
		Hence for any positive integer $k$, $k\polytope$ is a polytope given by inequalities of the form $kb_{ij} \leq z_i-z_j \leq kc_{ij}$.
		Moreover, $\polytope^{(1)}$ is a polytope given by inequalities of the form $b_{ij}+1 \leq z_i-z_j \leq c_{ij}-1$.
		Therefore, both $k\polytope$ and $\polytope^{(1)}$ are alcoved.
	\end{proof}
	For two lattice polytopes $\polytope$ and $Q$ in $\R^d$,
	set 
	$$\Ca(\polytope,Q) = \conv (\polytope \times \{0\} \cup Q \times \{1\})\subset \R^{d+1}.$$
	We say that $\Ca(\polytope,Q)$ is the \textit{Cayley polytope} of $\polytope$ and $Q$.
	\begin{lemma}
		Let $\polytope$ and $Q$ be alcoved polytopes in $\R^d$. Then $\Ca(\polytope,Q)$ has a regular unimodular triangulation. In particular, $\Ca(\polytope,Q)$ has the integer-decomposition property.
	\end{lemma}
	\begin{proof}
		This is~\cite[Lemma 4.15]{HPPS}, since alcoved polytopes have a type A root system.
	\end{proof}
	This directly implies the following result.
	\begin{corollary}
		\label{cor:Alcoved_onto}
		If $\polytope,Q \subset \R^d$ are alcoved polytopes, then the map
		\[
		\left(P \cap \Z^d\right) \times \left(Q \cap \Z^d \right) \twoheadrightarrow (P + Q) \cap \Z^d.
		\]
		is onto.
	\end{corollary}
	
	\begin{proof}
		We have that 
		\[
		2 \Ca(\polytope,Q) \cap \{(x_1,\ldots,x_{d+1}) \in \R^{d+1} : x_{d+1}=1\} =( \polytope+Q) \times \{1\}.
		\]
		Since $\Ca(\polytope,Q)$ has the integer-decomposition property, it follows  that every integer point in $2\Ca(\polytope,Q)$ can be written as a sum of two integer points in $\Ca(\polytope,Q)$. However, the only way we can get an integer point at height $1$ is if we add one integer point at height $0$ and at height $1$, i.e., one integer point belongs to $P$ and one belongs to $Q$, proving the claim.
	\end{proof}
	Now, we give a characterization on levelness of alcoved polytopes.
	
	\begin{proposition}
		\label{al_level}
		Let $\polytope \subset \R^d$ be an alcoved polytope and let $r=\textnormal{codeg}(\polytope)$.
		Then $\polytope$ is level if and only if
		for any integer $k \geq r$,
		it follows that  $(k\polytope)^{(1)}=(r\polytope)^{(1)}+(k-r)\polytope$. 
	\end{proposition}
	
	\begin{proof}
		First, assume that $\polytope$ is level.
		Then from the definition of levelness, 
		for any integer $k \geq r $, $\int(k\polytope) \cap \Z^d=\int(r \polytope) \cap \Z^d+(k-r)\polytope \cap \Z^d$.
		Hence one has $(k\polytope)^{(1)} =\conv((\int(r\polytope) \cap \Z^d)+(k-r)\polytope \cap \Z^d)=\conv(\int(r\polytope) \cap \Z^d)+\conv((k-r)\polytope \cap \Z^d)= (r\polytope)^{(1)}+(k-r)\polytope.$
		
		Conversely, assume that for any integer $k  \geq r$,
		$(k\polytope)^{(1)}=(r\polytope)^{(1)}+(k-r)\polytope$. 
		By Lemma~\ref{alcoved:lemma}, $(k\polytope)^{(1)}, (r\polytope)^{(1)}$ and  $(k-r)\polytope$ are alcoved polytopes.
		Hence by Corollary~\ref{cor:Alcoved_onto}, $\polytope$ is level.
	\end{proof}
	
	\begin{lemma}
		\label{lem:min}
		Let $\polytope \subset \R^d$ be a lattice polytope
		and let $r' \geq \textnormal{codeg}(\polytope)$ be an integer.
		Assume that there exists an integer $k > r'$
		such  that $(k\polytope)^{(1)}=(r'\polytope)^{(1)}+(k-r')\polytope$.
		Then for any integer $k > k' \geq r'$,
		we have $(k'\polytope)^{(1)}=(r'\polytope)^{(1)}+(k'-r')\polytope$.
	\end{lemma}
	
	\begin{proof}
		Assume that 
		there exists an integer	$k>k' \geq r'$ such that
		$(k'\polytope)^{(1)} \supsetneq  (r'\polytope)^{(1)}+(k'-r')\polytope$.
		In general, it follows that $(k'\polytope)^{(1)} \supset  (r'\polytope)^{(1)}+(k'-r')\polytope$.
		Then we have 
		$$(k\polytope)^{(1)} \supset (k'\polytope)^{(1)} +(k-k')\polytope \supsetneq  (r'\polytope)^{(1)}+(k-r')\polytope=(k\polytope)^{(1)}.$$
		Hence this is a contradiction.
	\end{proof}
	
	On levelness of dilated polytopes, the following theorem is known.
	\begin{theorem}[{\cite[Theorem 1.3.3]{BrunsGubeladzeTrung}}]
		\label{di_level}
		Let $\polytope$ be a lattice $d$-polytope.
		Then for any integer $k \geq d+1$, $k\polytope$ is level.
	\end{theorem}
	
	Now, we prove the following theorem about levelness of a product of alcoved polytopes.
	\begin{theorem}
		\label{thm:prod}
		Let $\polytope \subset \R^d$ and $Q \subset \R^{e}$ be alcoved polytopes.
		Suppose that $Q$ is level and  $r=\textnormal{codeg}(Q) \geq \dim \polytope+1$.
		Then $P \times Q \subset \R^{d+e}$ is level.
	\end{theorem}
	
	\begin{proof}
		By Theorem~\ref{di_level}, $r \polytope$ is level.
		Hence for any positive integer $k'$, one has $(k'r\polytope)^{(1)}=(r\polytope)^{(1)}+(k'-1)r\polytope$ from Proposition~\ref{al_level}.
		Therefore, by Lemma~\ref{lem:min}, it follows that 
		for any $k \geq r$, we obtain
		$(k\polytope)^{(1)}=(r\polytope)^{(1)}+(k-r)\polytope$.
		Since $Q$ is level, 
		for any $k \geq r$, we obtain
		$(kQ)^{(1)}=(rQ)^{(1)}+(k-r)Q$.
		Fix a positive integer $k \geq r$.
		Since
		$\text{int}(k(\polytope \times Q)) \cap \Z^{d+e}=(\text{int}(k\polytope) \cap \Z^d) \times (\text{int}(kQ) \cap \Z^{e}),$
		\begin{displaymath}
			\begin{aligned}
				(k(\polytope \times Q))^{(1)} &=(k\polytope)^{(1)} \times (kQ)^{(1)}\\
				&=((r\polytope)^{(1)}+(k-r)\polytope) \times ((rQ)^{(1)}+(k-r)Q)\\
				&\subset ((r\polytope)^{(1)} \times (rQ)^{(1)})+(k-r)(\polytope \times Q)\\
				&=(r(\polytope \times Q))^{(1)} +(k-r)(\polytope \times Q)\\
				&\subset (k(\polytope \times Q))^{(1)}.
			\end{aligned}
		\end{displaymath}
		Hence $\polytope \times Q$ is level.
	\end{proof}
	Now, we prove Theorem~\ref{thm:chainbound}.
	\begin{proof}[Proof of Theorem~\ref{thm:chainbound}]
		The order polytope of $\poset \cup C_{r}$ is the Cartesian product of $\order(\poset)$ and $\order(C_{r})$, which is the $r$-dimensional unimodular simplex. This simplex has codegree $r+1$.
		Hence by Theorem~\ref{thm:prod},  the claim now follows.
	\end{proof}
	
	Conversely, we consider posets all of whose connected components are level.
	In fact, these posets are always level. 
	
	\begin{theorem}
		\label{thm:prod2}
		Let $\poset$ be a poset on $d$ elements and $\poset_1,\ldots,\poset_m$ the connected components of $\Pi$.
		If each $\poset_i$ is level, then $\poset$ is level.
	\end{theorem}
	\rem{
		\begin{lemma}
			\label{lem:idp}
			Let $\Pc \subset \R^d$ be a level polytope of dimension $d$ which has the integer-decomposition property.
			Then for any positive integer $t$, $t\Pc$ is level.
		\end{lemma}
		\begin{proof}
			Let $\ell$ be the codegree of $\Pc$.
			Then we know that the codegree of $t\Pc$ equals
			$m:=\lceil \ell/t \rceil$.
			Let $k \geq m$ be positive integer and $x$ a lattice point of $\textnormal{int}(k(t\Pc))$.
			Since $\Pc$ is level,
			there exist $y \in \textnormal{int}(\ell \Pc) \cap \Z^d$ and $z_1,\ldots,z_{kt-\ell}  \in \Pc \cap \Z^d$ such that $x=y+z_1+\cdots+z_{kt-\ell}$.
			Then $y+z_1+\cdots+z_{mt-\ell} \in  \textnormal{int}(m(t\Pc)) \cap \Z^d$
			and $z_{mt-\ell+1},\ldots,z_{kt-\ell} \in t\Pc \cap \Z^d$.
			Hence $t\Pc$ is level.
		\end{proof}
	}
	
	Theorem~\ref{thm:prod2} follows from the following result:
	\begin{theorem}
		\label{thm:ProdOfLevel}
		Let $P \subset \R^d$ and $Q\subset \R^e$ be level polytopes. If either
		\begin{enumerate}
			\item $\codeg(Q)< \codeg (P)$ and $Q$ has the integer-decomposition property,
			\item $\codeg(P)< \codeg (Q)$ and $P$ has the integer-decomposition property,
			\item or if $\codeg(Q) = \codeg (P)$,
		\end{enumerate}
		then $P \times Q$ is level. 
	\end{theorem}
	
	\begin{proof}
		Let $r_P : = \codeg (P)$ and let $r_Q : = \codeg(Q)$. Without loss of generality, we assume $\codeg (Q)\leq \codeg(P)$. Then $r:= \codeg({P \times Q}) = \max \{r_P, r_Q \} = r_P$, since $\x = (\x_P, \x_Q) \in \int (P \times Q) \cap \Z^{d+e}$ implies that $\x_P \in \int(P) \cap \Z^d$ and that $\x_Q \in \int(Q) \cap \Z^e$. Let $(\x_P, \x_Q,h)\in \cone(P \times Q)\cap \Z^{d+e+1}$ with $h>r$. Then this point projects to points $(\x_P,h) \in \cone(P)\cap \Z^{d+1}$ and $(\x_Q,h) \in \cone(Q)\cap \Z^{e+1}$. Since $P$ is level, we have that 
		\[
		(\x_P,h) = (\x_P^{\circ},r) + (\tilde{\x}_P, h-r),
		\]
		where $(\x_P^{\circ},r) \in \int(\cone (P))\cap \Z^{d+1}$ and $(\tilde{\x}_P, h-r)\in \cone(P) \cap \Z^{d+1}$. Similarly, since $Q$ is level, we have that 
		\[
		(\x_Q,h) = (\x_Q^{\circ},r_Q) + (\tilde{\x}_Q, h-r_Q),
		\]
		where $(\x_Q^{\circ},r_Q) \in \int(\cone (Q))\cap \Z^{e+1}$ and $(\tilde{\x}_Q, h-r_Q)\in \cone(Q) \cap \Z^{e+1}$. If $\codeg(P) = \codeg(Q)$, we now get a decomposition
		\[
		(\x_P,\x_Q,h) = (\x^{\circ}_P, \x^{\circ}_Q, r) + (\tilde{\x}_P,\tilde{\x}_Q, h-r),
		\]
		where $(\x^{\circ}_P, \x^{\circ}_Q)  \in r \int(P \times Q) \cap \Z^{d+e}$ and where $(\tilde{\x}_P,\tilde{\x}_Q)\in (h-r)(P \times Q) \cap \Z^{d+e}$ proving levelness of $P\times Q$.

		So let's assume $\codeg(Q)<\codeg(P)$. Since $Q$ has the integer-decomposition property, we can express $(\tilde{\x}_Q, h-r_Q)$ as a sum of height $1$ elements, i.e., $(\tilde{\x}_Q, h-r_Q) = (\tilde{\x}_Q^{(1)}, r_P - r_Q) + (\tilde{\x}_Q^{(2)}, h-r_P)$ and thus obtain
		\[
		(\x_Q,h) = (\x_Q^{\circ},r_Q) + (\tilde{\x}_Q, h-r_Q) = {(\x_Q^{\circ} +\tilde{\x}_Q^{(1)} ,r_P)} + (\tilde{\x}_Q^{(2)}, h-r_P),
		\]
		where $(\x_Q^{\circ} +\tilde{\x}_Q^{(1)},r_P) \in \int (\cone(Q)) \cap \Z^{e+1}$ and $(\tilde{\x}_Q^{(2)}, h-r_P)\in \cone(Q) \cap Z^{e+1}$. Therefore, we can express $(\x_P, \x_Q, h)$ as
		\[
		(\x_P, \x_Q, h) = (\x_P^{\circ}, \x_Q^{\circ} +\tilde{\x}_Q^{(1)} , r) + (\tilde{\x}_P,\tilde{\x}_Q^{(2)}, h-r)
		\]
		where $(\x_P^{\circ}, \x_Q^{\circ} +\tilde{\x}_Q^{(1)}) \in r \int(P \times Q) \times \Z^{d+e}$ and where $(\tilde{\x}_P,\tilde{\x}_Q^{(2)}) \in (h-r)\int(P \times Q)\cap \Z^{d+e}$.
	\end{proof}
	
	\begin{remark}
		\label{rem:ProdOfLevel}
		In Theorem~\ref{thm:ProdOfLevel}, we really need the assumption that the polytope of lower codegree has the integer-decomposition property. Consider the following example, where 
		\[
		P = \conv \{(0,0,0),(1,0,0),(0,1,0),(1,1,2) \}
		\]
		and where
		\[
		Q = \conv\{(0,0),(1,0),(0,1) \}.
		\]
		Then $P$ does not have the integer-decomposition property, but it is Gorenstein and thus level. Moreover, $Q$ has the integer-decomposition property and it is level, since it is Gorenstein. However, the product $P \times Q$ is not level.
		\begin{proof}
			We first remark that $\codeg({Q}) = 3$ and $\codeg({P}) = 2$. Hence, $\codeg (P \times Q) = 3$. There are exactly $4$ integer points in the interior of $3 (P \times Q)$, namely
			\[
			\int (3(P \times Q))\cap \Z^5 = \{(1,1,1,1,1), (1,2,1,1,1),(2,1,1,1,1),(2,2,3,1,1) \}.
			\]
			However, the integer point $(2,2,2,2,1) \in \int (4 (P \times Q))\cap \Z^{5}$ cannot be written as a sum
			\[
			(2,2,2,2,1) = \x_1 + \x_2
			\] 
			of points $\x_1 \in \int (3(P \times Q))\cap \Z^5$ and $\x_2 \in (P \times Q) \cap \Z^5$. Thus, $P \times Q$ is not level.
		\end{proof}
		
	\end{remark}
	
	We end this article with the following criterion for levelness:
	\begin{lemma}
		\label{lem:unimodularcoverwithinternalface}
		If a lattice $d$-polytope $\polytope$ has a covering by unimodular simplices such that every interior face of such a simplex contains an interior sub-face of dimension $\codeg(\polytope)-1$, then $\polytope$ is level.
	\end{lemma}
	
	\begin{proof}
		Let $\b \in \int (\cone(\polytope))\cap \Z^{d+1}$ and let $r : = \codeg(P)$. We will prove levelness by showing that every such $\b$ can be written as an integral combination of an interior integer point on height $r$ with integer points on height $1$.

		The unimodular simplices in the covering of $\polytope$ give rise to a covering of $\cone(\polytope)$ by unimodular cones. Therefore, $\b$ is in at least one such unimodular cone. Let $\Delta = \conv \{\v_1,\v_2,\dots,\v_{d+1} \}$ be the corresponding unimodular simplex and let $\cone(\Delta)$ be the cone over $\Delta$. Then there is a unique representation 
		\begin{equation}
			\label{eq:xwithrespecttoDelta}
			\b= \sum_{i = 1}^{d+1}\lambda_{i}(\v_i,1),
		\end{equation}
		where $\lambda_i \in \Z_{\geq 0}$. There are two cases: If $\lambda_i=0$ for some $i\in I$ in an index set $I$, then this means that $\b$ is contained an $(d+1 - \#I)$-dimensional face, which has to be an interior face. Then, for $J :=\{1,\dots,d+1 \}\setminus I$, the face $\conv\{ \v_j\}_{j \in J}$ gives rise to a ($\#J$)-dimensional cone containing $\b$. Hence, by assumption there is a subset $R \subset J$ with $\# R = r$, such that the point $\x = \sum_{i \in R} (\v_i,1)$ is a point in $\int ( \cone(\polytope))\cap \Z^{d+1}$. Combining this with Equation (\ref{eq:xwithrespecttoDelta}), we obtain
		\[
		\b = \x + \left(\b - \x\right),
		\]
		where $\b- \x \in \cone(\polytope) \cap \Z^{d+1}$, which proves levelness,
		
		If $\lambda_{i} \neq 0$ for all $i\in \{ 1,2,\dots, d+1\}$, then we can set $\x = \sum_{i \in J}(\v_i,1)$, where $J$ is an interior ($r-1$)-dimensional face. The claim now follows analogously.
	\end{proof}
	It would be interesting to see a poset interpretation of this lemma.
	\bibliographystyle{amsalpha}
	\bibliography{refs}

\providecommand{\bysame}{\leavevmode\hbox to3em{\hrulefill}\thinspace}
\providecommand{\MR}{\relax\ifhmode\unskip\space\fi MR }
\providecommand{\MRhref}[2]{%
  \href{http://www.ams.org/mathscinet-getitem?mr=#1}{#2}
}
\providecommand{\href}[2]{#2}
\begin{thebibliography}{EHHSM15}

\bibitem[Bel58]{Bellman}
Richard Bellman, \emph{On a routing problem}, Quart. Appl. Math. \textbf{16}
  (1958), 87--90. \MR{0102435}

\bibitem[BG09]{BrunsGubeladze}
Winfried Bruns and Joseph Gubeladze, \emph{Polytopes, rings, and {$K$}-theory},
  Springer Monographs in Mathematics, Springer, Dordrecht, 2009. \MR{2508056}

\bibitem[BGT02]{BrunsGubeladzeTrung}
Winfried Bruns, Joseph Gubeladze, and Ng\^o~Vi\^et Trung, \emph{Problems and
  algorithms for affine semigroups}, Semigroup Forum \textbf{64} (2002), no.~2,
  180--212. \MR{1876854}

\bibitem[BH17]{BallettiHigashitani}
Gabriele Balletti and Akihiro Higashitani, \emph{Universal inequalities in
  ehrhart theory}, arxiv (2017).

\bibitem[BR15]{CCD}
Matthias Beck and Sinai Robins, \emph{Computing the continuous discretely},
  second ed., Undergraduate Texts in Mathematics, Springer, New York, 2015,
  Integer-point enumeration in polyhedra, With illustrations by David Austin.
  \MR{3410115}

\bibitem[EHHSM15]{EneHerzogHibi}
Viviana Ene, J\"urgen Herzog, Takayuki Hibi, and Sara Saeedi~Madani,
  \emph{Pseudo-{G}orenstein and level {H}ibi rings}, J. Algebra \textbf{431}
  (2015), 138--161. \MR{3327545}

\bibitem[Ehr62]{Ehrhart}
Eug\`ene Ehrhart, \emph{Sur les poly\`edres rationnels homoth\'etiques \`a
  {$n$}\ dimensions}, C. R. Acad. Sci. Paris \textbf{254} (1962), 616--618.
  \MR{0130860}

\bibitem[HH16]{HibiHigashitani}
Takayuki Hibi and Akihiro Higashitani, \emph{Integer decomposition property of
  free sums of convex polytopes}, Ann. Comb. \textbf{20} (2016), no.~3,
  601--607. \MR{3537921}

\bibitem[Hib87]{HibiGorenstein}
Takayuki Hibi, \emph{Distributive lattices, affine semigroup rings and algebras
  with straightening laws}, Commutative algebra and combinatorics ({K}yoto,
  1985), Adv. Stud. Pure Math., vol.~11, North-Holland, Amsterdam, 1987,
  pp.~93--109. \MR{951198}

\bibitem[Hib88]{HibiStraightening}
\bysame, \emph{Level rings and algebras with straightening laws}, J. Algebra
  \textbf{117} (1988), no.~2, 343--362. \MR{957445}

\bibitem[Hib90]{HibiInequalities}
\bysame, \emph{Some results on {E}hrhart polynomials of convex polytopes},
  Discrete Math. \textbf{83} (1990), no.~1, 119--121. \MR{1065691}

\bibitem[HKN16]{Spanning}
Johannes Hofscheier, Lukas Katth\"an, and Benjamin Nill, \emph{Ehrhart theory
  of spanning lattice polytopes}, arxiv (2016).

\bibitem[Hoc72]{Hochster}
M.~Hochster, \emph{Rings of invariants of tori, {C}ohen-{M}acaulay rings
  generated by monomials, and polytopes}, Ann. of Math. (2) \textbf{96} (1972),
  318--337. \MR{0304376}

\bibitem[HPPS14]{HPPS}
Christian Haase, Andreas Paffenholz, Lindsay~C. Piechnik, and Francisco Santos,
  \emph{Existence of unimodular triangulations --- positive results},
  arXiv:1405.1687 (2014).

\bibitem[HY18]{HigashitaniYanagawa-Level}
Akihiro Higashitani and Kohji Yanagawa, \emph{Non-level semi-standard graded
  {C}ohen-{M}acaulay domain with {$h$}-vector {$(h_0, h_1,h_2)$}}, J. Pure
  Appl. Algebra \textbf{222} (2018), no.~1, 191--201. \MR{3681002}

\bibitem[KO17]{KohlOlsen}
Florian Kohl and McCabe Olsen, \emph{Level algebras and $s$-lecture hall
  polytopes}, arxiv:1710.10892 (2017).

\bibitem[Miy17]{Miyazaki}
Mitsuhiro Miyazaki, \emph{On the generators of the canonical module of a {H}ibi
  ring: a criterion of level property and the degrees of generators}, J.
  Algebra \textbf{480} (2017), 215--236. \MR{3633306}

\bibitem[MS05]{MillerSturmfels}
Ezra Miller and Bernd Sturmfels, \emph{Combinatorial commutative algebra},
  Graduate Texts in Mathematics, vol. 227, Springer-Verlag, New York, 2005.
  \MR{2110098}

\bibitem[Sch86]{schrijver1986}
A.~Schrijver, \emph{Theory of linear and integer programming},
  Wiley-Interscience Series in Discrete Mathematics, John Wiley \& Sons Ltd.,
  Chichester, 1986, A Wiley-Interscience Publication. \MR{MR874114 (88m:90090)}

\bibitem[Sch03]{SchrijverCombOptA}
Alexander Schrijver, \emph{Combinatorial optimization. {P}olyhedra and
  efficiency. {V}ol. {A}}, Algorithms and Combinatorics, vol.~24,
  Springer-Verlag, Berlin, 2003, Paths, flows, matchings, Chapters 1--38.
  \MR{1956924}

\bibitem[Sta77]{Stanley-CM-complexes}
Richard~P. Stanley, \emph{Cohen-{M}acaulay complexes}, Higher combinatorics
  ({P}roc. {NATO} {A}dvanced {S}tudy {I}nst., {B}erlin, 1976), NATO Adv. Study
  Inst. Ser., Ser. C: Math. and Phys. Sci., 31, Reidel, Dordrecht, 1977,
  pp.~51--62.

\bibitem[Sta78]{Stanley-HilbertFunctions}
\bysame, \emph{Hilbert functions of graded algebras}, Advances in Math.
  \textbf{28} (1978), no.~1, 57--83. \MR{0485835}

\bibitem[Sta80]{StanleyNonnegativity}
\bysame, \emph{Decompositions of rational convex polytopes}, Ann. Discrete
  Math. \textbf{6} (1980), 333--342, Combinatorial mathematics, optimal designs
  and their applications (Proc. Sympos. Combin. Math. and Optimal Design,
  Colorado State Univ., Fort Collins, Colo., 1978). \MR{593545}

\bibitem[Sta86]{StanleyOrderPoly}
\bysame, \emph{Two poset polytopes}, Discrete Comput. Geom. \textbf{1} (1986),
  no.~1, 9--23. \MR{824105}

\bibitem[Sta91]{StanleyInequalities}
\bysame, \emph{On the {H}ilbert function of a graded {C}ohen-{M}acaulay
  domain}, J. Pure Appl. Algebra \textbf{73} (1991), no.~3, 307--314.
  \MR{1124790}

\bibitem[Sta96]{StanleyGreenBook}
\bysame, \emph{Combinatorics and commutative algebra}, second ed., Progress in
  Mathematics, vol.~41, Birkh\"auser Boston, Inc., Boston, MA, 1996.
  \MR{1453579}

\bibitem[Sta09]{Stapledon1}
Alan Stapledon, \emph{Inequalities and {E}hrhart {$\delta$}-vectors}, Trans.
  Amer. Math. Soc. \textbf{361} (2009), no.~10, 5615--5626. \MR{2515826}

\bibitem[Sta12]{StanleyEC1}
Richard~P. Stanley, \emph{Enumerative combinatorics. {V}olume 1}, second ed.,
  Cambridge Studies in Advanced Mathematics, vol.~49, Cambridge University
  Press, Cambridge, 2012. \MR{2868112}

\bibitem[Sta16]{Stapledon2}
Alan Stapledon, \emph{Additive number theory and inequalities in {E}hrhart
  theory}, Int. Math. Res. Not. IMRN (2016), no.~5, 1497--1540. \MR{3509934}

\bibitem[Zie95]{Ziegler}
G\"unter~M. Ziegler, \emph{Lectures on polytopes}, Graduate Texts in
  Mathematics, vol. 152, Springer-Verlag, New York, 1995. \MR{1311028}

\end{thebibliography}

\end{document}